\documentclass[10pt]{amsart}
\usepackage{amsfonts}
\usepackage{amsbsy}
\usepackage{tikz}
\usepackage{amsmath}
\usepackage{amssymb}
\usepackage{amsthm}
\usepackage{syntonly}
\usepackage{amscd} 
\usepackage[colorlinks,linkcolor=brown, anchorcolor=blue, citecolor=green ]{hyperref}
\usepackage{graphicx} 

\usepackage{mathrsfs} 

\usepackage{comment} 

\usepackage{bbm} 

\newcommand{\R}{\mathbb{R}} 
\newcommand{\Z}{\mathbb{Z}}
\newcommand{\Q}{\mathbb{Q}}
\newcommand{\C}{\mathbb{C}}
\newcommand{\N}{\mathbb{N}}

\newcommand{\A}{\mathscr{A}}
\newcommand{\B}{\mathscr{B}}

\newcommand{\PPP}{\mathcal{P}}
\newcommand{\OO}{\mathcal{O}}

\newcommand{\ep}{\varepsilon}

\newtheorem{thm}{Theorem}[section]

\newtheorem{lem}[thm]{Lemma}
\newtheorem{defi}[thm]{Definition}

\newtheorem{prop}[thm]{Proposition}

\newtheorem{coro}[thm]{Corollary}

\DeclareMathOperator{\SL}{SL}
\DeclareMathOperator{\SU}{SU}

\DeclareMathOperator{\Vol}{Vol}
\DeclareMathOperator{\diag}{diag}
\DeclareMathOperator{\Gal}{Gal}
\DeclareMathOperator{\Lie}{Lie}
\DeclareMathOperator{\Ad}{Ad}
\DeclareMathOperator{\Res}{Res}
\DeclareMathOperator{\Nm}{Nm}
\DeclareMathOperator{\dist}{dist}
\DeclareMathOperator{\Perm}{Perm}

\newenvironment{psmallmatrix}
  {\left(\begin{smallmatrix}}
  {\end{smallmatrix}\right)}

\begin{document}

\title[Translates of maximal torus orbits]{Limiting distribution of translates of the orbit of a maximal $\mathbb{Q}$-torus from identity on $\SL_N(\mathbb{R})/\SL_N(\mathbb{Z})$}
\author[R.Zhang]{Runlin Zhang }

\address{Department of Mathematics, The Ohio State University\\ 231 W 18th Avenue\\
Columbus, Ohio 43210-1174}
\email{zhang.4749@osu.edu}
\date{Sep 10, 2018}

\begin{abstract}
    Given a closed orbit of the real points of a maximal $\Q$-torus in $\SL(N,\R)/\SL(N,\Z)$, it naturally carries a possibly infinite Haar measure. We classify all possible limit measures of it when translated by a sequence of elements from $\SL(N,\R)$. This is a natural extension of Shapira and Zheng's work where only $\Q$-split torus are considered.
\end{abstract}


\maketitle

\tableofcontents

\section{Introduction}
\subsection{Motivation}
Observed by Duke-Rudnick-Sarnak \cite{DukRudSar93}, counting integer points on a homogeneous affine variety of the form $G/H$ for $G, H$ both reductive linear algebraic groups defined over $\Q$ is related to the limiting measure of translates of the $H(\R)-$invariant measure $\mu_{H(\R)G(\Z)/G(\Z)}$ in the quotient space $G(\R)/G(\Z)$. Assume $G$ to have no rational characters, equivalently, $G(\R)/G(\Z)$ has finite $G(\R)$ measure. Then the latter problem can be approached by theorems about unipotent-invariant probability measures known as Ratner's theorems(\cite{Rat91}) for the following reason:

Given a sequence $g_k$ in $G(\R)$ and take $\nu$ to be a limit point of $(g_k)_*\mu_{H(\R)G(\Z)/G(\Z)}$. Then $\nu$ is invariant under any limit points of $\exp(\Ad(g_k) v_k)$ for some sequence $v_k\in \text{Lie}(H(\R))$. Now suppose $\Ad(g_k) v_k/|v_k|$ diverges to infinity but $\Ad(g_k) v_k$ converges to a nonzero vector, then $v_k$ necessarily goes to 0 and all eigenvalues of $v_k$ goes to 0. But $\Ad(g_k) v_k$ has the same eigenvalues as $v_k$ hence any limit of $\Ad(g_k) v_k$ must have all eigenvalues vanishing, i.e. a nilpotent matrix. For this reason, the limit measure $\nu$ has to be invariant under exponential of some nontrivial nilpotent matrix, i.e. a nontrivial unipotent element. Of course, another important issue here is that $\nu$ needs to be a nonzero measure.

Indeed, this approach has been carried out successfully by Eskin-Mozes-Shah in \cite{EskMozSha96} and \cite{EskMozSha97}. In some special cases one may even approach by mixing as in the work of Eskin-McMullen \cite{EskMcM93}. However, one assumption that has always been made was that $H(\R)/H(\Z)$ has finite covolume. The barrier was crossed by Oh-Shah \cite{OhSha14}, where they allow $H$ to be a $\Q$-split torus in $\SL_2(\Q)$. A further refinement is obtained by Kelmer-Kontorovich \cite{KelKon18} using different methods, which further allows one to identity a lower order "distribution".  This was generalized later to the case $\SL_N(\Q)$ with $H$ maximal $\Q$-split torus by Shapira-Zheng \cite{ShaZhe18} where they give a complete classification of all possible limiting measures. As a corollary, they also classify all limiting measures for $H$ reductive and contains a maximal $\Q$-split torus. In the present work their results are further generalized to all maximal tori defined over $\Q$ in $\SL_N(\Q)$, which is the case when translating a closed orbit of the diagonal torus by a theorem of Tomanov-Weiss \cite{TomWei03}. With additional (unnecessary) assumptions on the sequence which is used to translate the measure, Shah had obtained a similar result(unpublished). 
The emphasize here would be a complete classification.

Let us remark that results in \cite{OhSha14}, \cite{KelKon18} are effective (\cite{OhSha14} treats other non-arithmetic lattices and \cite{KelKon18} treats moreover the case of thin groups). And methods from \cite{EskMcM93} can be made effective to treat the case when $H$ is a symmetric subgroup with finite covolume, for instance, the case when $H$ is an anisotropic torus in $\SL_2(\Q)$. 
This has been carried out by Benoist-Oh in \cite{BenOh12}. However, it seems that no such effective results is known for translation of maximal torus in $\SL_N(\Q)$ for $N$ at least 3. And the proof presented here as well as the proof in \cite{ShaZhe18} makes crucial use of Ratner's theorem, which only yields noneffective results.

A similar problem of translating a torus orbit on adelic quotients is considered by Zamojski in his thesis \cite{Zam10}. 
The situation is very different from ours as no unipotent invariance is present in their work.

One should also mention the work of Richard and Zamojski \cite{RicZam17}, where they consider translation of $\Omega x$ where $\Omega$ is a bounded open in $H(\R)$ and $x\in G(\R)/\Gamma$ is arbitrary. It might be possible to use their methods as a substitute of \cite{EskMozSha96}, but we did not try.

\subsection{Convergence of measures}

In this paper, we are concerned with convergence of possibly infinite measures(see \cite{ShaZhe18} Section 3 for a more thorough introduction). For $X$, a locally compact second countable Hausdorff space, we let $\mathcal{M}(X)$ be the set of locally finite measures on $X$ and $C_c(X)$ be the set of compactly supported continuous real-valued functions. For $\mu \in \mathcal{M}(X)$ and $f\in C_c(X)$, we use $(f,\mu)$ to denote $\int f d\mu$, the integral of $f$ with respect to $\mu$.

Two nonzero measures $\mu$ and $\nu$ from $\mathcal{M}(X)$ 
are said to be equivalent if there is a positive real number 
$a>0$ such that $\mu=a \nu$. 
The set of equivalence classes are denoted by $\mathbb{P}\mathcal{M}(X)$. This set is equipped with a topology such that a sequence
$\{[\mu_i]\}_{i\in \N} \subset \mathbb{P}\mathcal{M}(X)$
converges to 
$[\nu]\in \mathbb{P}\mathcal{M}(X)$ 
if and only if one of the following equivalent condition is satisfied (see Proposition 3.3 in \cite{ShaZhe18}):

\begin{enumerate}
    \item  For all $f_1,f_2\in C_c(X)$, compactly supported continuous functions on $X$ such that $(f_2,\nu)\neq 0$ we have
\begin{equation*}
    \lim_{i\to \infty}
    \frac{ \int f_1(x) d\mu_i(x)  }  
    {\int f_2(x) d\mu_i(x)} =\frac{\int f_1(x) d\nu(x)}{\int f_2(x) d\nu(x)}
\end{equation*}
    \item There exists a sequence of positive real number $a_i >0$ such that for all $f\in C_c(X)$, $a_i (f,\mu_i)\to ( f, \nu)$.
\end{enumerate}

Note that if all measures $\mu_i$, $\nu$ above are already probability measures, we retrieve the classical weak-$*$ convergence. But a sequence of probability measures $\mu_i \to 0$ in the weak-$*$ topology may converge to a nonzero measure in this new topology.  The phenomenon we shall see below is that class of infinite measures may converge to a finite one. 

We also want to remark here that although we are proving statements about possibly infinite measures in this paper, the proof goes via chopping them into finite measures and summing together the results for those finite measures. 

We also define, for a finite nonzero measure $\mu$,  $\hat{\mu}$ to be the unique probability measure in the equivalence class of $\mu$, i.e. $[\mu]=[\hat{\mu}]$ and $\hat{\mu}(X)=1$.

\subsection{Generic case}

Let $G= \SL_N$ with standard $\Q$-structures and $\Gamma \subset \SL _N(\Q)$ an arithmetic lattice. We fix a right invariant Riemannian metric on $G(\R)$ such that with respect to this metric $G(\R)/\Gamma$ has volume $1$. We let $\mu_{G}$ be the unique $G(\R)$-invariant probability measure on $X:=G(\R)/\Gamma$.

Let $H \leq G$ be a reductive subgroup defined over $\Q$. Then $H(\R)\Gamma/\Gamma \subset G(\R)/\Gamma$  is a closed submanifold and hence $H(\R)\Gamma/\Gamma$ supports a $H(\R)$-invariant locally finite measure induced from the metric of $G(\R)/\Gamma$, which we will denote by $\mu_H$. 

For a sequence $\{g_i\}_{i\in \N} \subset G(\R)$, we have the following criterion as to when $(g_i)_*[\mu_T]$ converges to $[\mu_G]$.

\begin{thm}\label{theorem0}
Given a maximal $\Q$-torus $T$ of $G= \SL_N$ with the standard $\Q$-structure. Let $\Gamma \subset G(\Q)$ be an arithmetic lattice. Assume $\{g_i\} \subset G(\R)$ diverges when projecting to $G(\R)/Z_GS(\R)$ for every nontrivial 
$\Q$-subtorus $S$ of $T$, then $(g_i)_*[\mu_T]$ converges to $[\mu_G]$.
\end{thm}

The theorem is illustrated by two examples below. 
We will give rigorous proofs for them in Section \ref{section7'}.

We then generalize Theorem \ref{theorem0} to other number fields.

\begin{thm}\label{theorem2}
Let $G'=R_{M/\Q}\SL_N$ be the restriction of scalars of the $M$-group $\SL_N$, for $M$ a  number field. Let $T'$ be a maximal $\Q$-torus in $G'$ and $\Gamma'\subset G'(\Q)$ be an arithmetic lattice. Given a sequence $\{g_i\} \subset G'(\R)$ that diverges in $G'(\R)/Z_{G'}(S')(\R)$ for every nontrivial $\Q$-subtorus $S' \subset T'$, we have $\lim_i (g_i)_{*}[\mu_{T'}]=[\mu_{G'}]$.
\end{thm}

\subsection{Intermediate cases}

Back to the case $G=\SL_N$ defined over $\Q$ and $T$ a maximal $\Q$-torus.
To handle the case when the limit measure of $(g_i)_*[\mu_T]$ is not the $G(\R)$-invariant measure, we make the following definition.

For a maximal
$\Q$-torus $T$ and a $\Q$-subtorus $S_0$, we say a sequence $\{g_i\}_{i\in \N} \subset G(\R)$ is $(S_0, T)$-clean if the following holds:\\

\textit{
 $\{g_i\}\subset Z_G(S_0)(\R)$, the real points of the centralizer of $S_0$ in $G$. And for any $\Q$-subtorus $S$ of $T$ properly containing $S_0$, 
 $\{g_i\}$ diverges in $G(\R)/Z_G(S)(\R)$.
}\\

Fixing a maximal $\Q$-torus $T$ and arbitrary sequence $\{g_i\} \subset G(\R)$, by passing to a subsequence and modifying from left by a bounded sequence, we may always assume that $\{g_i\}$ is $(S_0,T)$-clean for some $\Q$-subtorus $S_0$ of $T$. This is because $S_0$ satisfies $Z(Z_G(S_0))=S_0$ and there are only finitely many of them in $T$ (For details see Section \ref{section6}).
Also note that the centralizer of a torus in a reductive group (our $G$ is semisimple and hence reductive) is reductive (see \cite{Spr98}, Corollary 7.6.4), 
so $\mu_{H_0}$ is well-defined for $H_0 =Z_G(S_0)$.

\begin{thm}\label{theorem1}
Assume we are given a maximal $\Q$-torus $T$ and a $\Q$-subtorus $S_0$. Assume the sequence $\{g_i\}_{i\in \N} \subset G(\R)$ is $(S_0, T)$-clean.
Let $H_0:=Z_G(S_0)$ then $(g_i)_*[\mu_T]$ converges to $[\mu_{H_0}]$.
\end{thm}

This theorem will be deduced from \ref{theorem2} after we classify all possible $H_0$, which we will carry out in Section \ref{section6}. \\

Now we turn to the problem of translating Haar measures supported on $H(\R)\Gamma/\Gamma$ for a reductive $\Q$-subgroup $H$ that contains a maximal $\Q$-torus.

Let us observe that any connected reductive subgroup $H_1$ in $G=\SL_N$ containing a maximal torus is of the form $Z_G(S_1)$ with $S_1=Z(H_1)$ being a torus. We assume $H_1$ is defined over $\Q$ then so is $S_1$. The following results about translating the closed orbit $H_1(\R)\Gamma/\Gamma \subset G(\R)/\Gamma$ will be deduced from our Theorem \ref{theorem1}.

For a $\Q$-torus $S_1$ and a $\Q$-subtorus $S_0$,  
a sequence $\{g_i\}_{i\in \N} \subset G(\R)$ is said to be $(S_0,S_1)$-clean if  the following condition is satisfied:\\

 \textit{
 $\{g_i\}\subset Z_G(S_0)(\R)$ and for any $\Q$-subtorus $S$ of $S_1$ strictly containing $S_0$, $\{g_i\}$ diverges in $G(\R)/Z_G(S)(\R)$.
 }\\

 Fixing any $\Q$-torus $S_1$ of $G$ and any sequence $\{g_i\}$ of $G(\R)$, by passing to a subsequence and modifying from left by a bounded sequence in $G(\R)$, we may always assume 
  $\{g_i\}$ is $(S_0,S_1)$-clean for some $\Q$-subtorus $S_0$ of $S_1$. The reason is similar to why we can assume $(S_0,T)$-clean above in Theorem \ref{theorem1}.

\begin{thm}\label{theorem3}
Assume we are given two $\Q$-tori $S_0\subset S_1$ and a sequence $\{g_i\} \subset G(\R)$ that is $(S_0,S_1)$-clean. 
Let $H_i$ be $Z_G(S_i)$ for $i=0,1$, then $(g_i)_*[\mu_{H_1}]$ converges to $[\mu_{H_0}]$.
\end{thm}

\subsection{A counting problem}

We then treat a counting problem by applying the method of \cite{DukRudSar93} (c.f.\cite{EskMcM93}) combined with Theorem \ref{theorem2} and some volume computation. 

We take $M$ to be a number field, i.e. a finite extension of $\Q$. Assume the extension degree is $m_0$. Fix $p(x)\in \OO_M[x]$, a polynomial of degree N. Assume in $M[X]$, $p(x)=p_0(x)p_1(x)\cdot ...\cdot p_{a_0}(x)$ with $p_0$ of degree $l_0$ completely reducible and all other $p_i$'s are irreducible. We require this polynomial to have distinct roots in a fixed algebraic closure $\overline{\Q}$ of $\Q$. We let $X_p(\OO_M)$ to be the collection of $N$-by-$N$ matrices with $\OO_M$ coefficients whose characteristic polynomial is equal to $p(x)$. 

Let $\{\alpha_1,...,\alpha_{r_0}\}$ be all real embeddings of $M$ and 
$\{\beta_1,...\beta_{s_0}\} \sqcup \{\overline{\beta_1},...\overline{\beta_{s_0}}\}$ be all conjugate pairs of complex embeddings.  So $r_0+2s_0$ is equal to $m_0$, the extension degree of $M/\Q$.

We may now identify the set $X_p(\OO_M)$ with 
$$X'_p(\Z) := \{(A_1,...,A_{r_0}, B_1,...,B_{s_0}) \,\vert\, \exists A \in X_p(\OO_M), \, A_i= \alpha_i(A),\, B_j=\beta_j(A) \},$$
 which is a discrete subset in $M_N(\R)\times...\times M_N(\R) \times M_N(\C)\times...\times M_N(\C)$ with $r_0$ copies of $M_N(\R)$ and $s_0$ copies of $M_N(\C)$ where we define a norm $||(x_1,...,x_{r_0+s_0})||$ $:=(||x_1||^2+...+||x_{r_0+s_0}||^2)^{1/2}$ and we use here the Euclidean norm for each individual matrix. With this norm we let $B_R$ be the ball of radius $R$.

\begin{thm}\label{theorem4}
There exists a constant $c_p > 0$ such that  
$$\lim_{R\to + \infty} \frac{|X'_p(\Z)\cap B_R|}{c_pR^{m_0N(N-1)/2} (\log{R})^{a_0+l_0-1} } =1$$
\end{thm}

\subsection{Strategy of the proof}

As mentioned before, our proof ultimately relies on the description of finite unipotent-invariant measures by Ratner \cite{Rat91} and the linearization technique developed by Dani-Margulis \cite{DanMar93}. These results enter into the proof via the work of Eskin-Mozes-Shah \cite{EskMozSha96}.
Elsewhere we follow the framework of Shapira-Zheng \cite{ShaZhe18}, but several additional difficulties arises when the maximal torus $T$ is not $\Q$-split.

In Section \ref{section2} we will define certain families of convex polytopes depending on the sequence $g_i$ that lie in the Lie$(T_s(\R))$, where $T_s$ denotes the $\Q$-split part of the maximal torus $T$. The main result in this section is that we show this region corresponds to points that come back to a compact subset of $G(\R)/\Gamma$ under translation by $g_i$. This section is written in a way that is useful for both $G=\SL_N$ and $G'=R_{M/\Q}G$.

In Section \ref{section3} we adapt the method of \cite{ShaZhe18} to our situation, using an auxiliary graph and properties of the convex polytopes to show that the polytopes defined in Section \ref{section2} contains another family of polytopes with volume keeping asymptotically the same.

In Section \ref{section4} we show that for all bounded pieces in this new family of polytopes, their tranlates under $g_i$ get equidistributed towards $\mu_G$, the probability Haar measure on $G(\R)/\Gamma$. This implies our Theorem \ref{theorem0} .

In Section \ref{section5} we treat the case $G'=R_{M/\Q}G$. 
There are several places where the generalization is not so straightforward. Theorem \ref{theorem2} will be proved here.
 
And in Section \ref{section6} we treat the remaining cases of Theorem \ref{theorem1} by reducing it to  Theorem \ref{theorem2}. For this purpose, we classifies all possible $S_0$'s that are allowed to appear in the definition of $(S_0,T)$-clean, which partially generalizes the work of Tomanov \cite{Tom00}. 

In Section \ref{section7'} we give some examples and compare our theorem to those of \cite{EskMozSha96} and \cite{ShaZhe18}.
And in Section \ref{section7} we deduce Theorem \ref{theorem3} from Theorem \ref{theorem1}. 

In the last Section \ref{section8} we prove Theorem \ref{theorem4}. This is based on the last section of \cite{ShaZhe18} and some additional computations to handle the case when the torus is not $\R$-split.

\section{The polytope of non-divergence}\label{section2}

\subsection{A concrete form of maximal torus}
We begin with some basic set-ups.
Let $G=\SL_N$ viewed as defined over $M$, a finite field extension of $\Q$. $T$ is a maximal torus in $G$ defined over $M$. Any such torus will be conjugate(over $M$) to one that we describe below (see for instance Lemma 6.2, \cite{LinWei01} when $M=\Q$ and the argument applies equally well to the general case) and we shall assume throughout the paper $T$ takes such a form.

Take $L_1,...,L_{a_0}$ field extensions of $M$ of finite degree. Take a set of representatives for each $\Gal(\overline{\Q}/M)/\Gal(\overline{\Q}/L_i)$ and number them as 
$\{\sigma^i_1,\sigma^i_2,..., \sigma^i_{l_i}\}$ with $l_i$'s the degree of extensions of $L_i/M$. 
Choose basis
$\{v_1^i,...,v_{l_i}^i\}$ 
in $\OO_{L_i}$ for the $M$-vector space $L_i$. 
For $x\in L_i$, we let $A^i(x)$ be the diagonalized $l_i$-by-$l_i$ matrix $\text{diag}\{\sigma^i_1x,...,\sigma^i_{l_i}x\}$. 
Consider the set of $l_i$-by-$l_i$ matrices 

$$\{(\sigma^i_{\xi}v^i_{\zeta})_{\xi,\zeta}^{-1} A^i(x) (\sigma^i_{\xi}v^i_{\zeta})_{\xi,\zeta} \vert \, \text{Nm}_{L_i/M}(x)=1\}.$$ 

By checking that they are fixed by $\Gal(\overline{\Q}/M)$-action we see they are all matrices of coefficients in $M$. Taking its Zariski closure we arrive at a $M$-torus in $\SL_{l_i}$ of dimension $l_i-1$, hence maximal in $\SL_{l_i}$ which We denote by $T^i_a$ as it is $M$-anisotropic.  We let $l_0$ be a number such that $l_0+l_1+...+l_{a_0}=N$ and assume it is non-negative. We now define $T$ by specifying its $M$-split part $T_s$ and $M$-anisotropic part $T_a$.

$$T_s=\{A=\diag(x_1,...,x_{l_0}, y_1I_{l_1},...,y_{a_0}I_{l_{a_0}}  ) \,\vert\, \det{A}=1\, \}$$
and 
$$T_a=\{ A=\text{diag}(I_{l_0},A_1,...,A_{a_0}  )  \, \vert \, A_i\in T^i_a \,\}.$$

\subsection{Definition of polytope}

Take $M=\Q$ for the discussion in this subsection. When $T=T_a$, the orbit $T(\R)\Gamma/\Gamma$ is compact and nothing needs to be done. Otherwise 
$T(\R)\Gamma/\Gamma$ would have a divergent part. We expect that after translating by $B\in G(\R)$, some part of the orbit, say $\exp(t)\Gamma$, that is far away will be brought closer. 
In light of Mahler's criterion, this means $B\exp{(t)}v$ will be bounded away from the origin for all non-zero integer vectors in $\R^N$ and also those in $\bigwedge^d \R^N$. This motivates the following definition(c.f. Definition 4.1, \cite{ShaZhe18} ).

To avoid unnecessary complications, we introduce some notations.
We view $\{1,2,...,N\}$ as $\{1,...,l_0\} \sqcup \{ 1,...,l_1   \} \sqcup...\sqcup \{1,...,l_{a_0}\} $. We use $[l_i]$ as a shorthand for $\{1,...,l_i\}$. We use the notation $b^{(0)},...,b^{(a_0)}$ when there is a possible confusion about from which set this "$b$" comes from. 
We also let $\A_0$ to be 
$2^{\{1,...,l_0\}} \times \{\emptyset, \{1,...,l_1\}\} \times...\times \{\emptyset, \{1,...,l_{a_0}\}\} $, a collection of subsets of $\{1,...,N\}$. 
$\{e_1,...,e_{l_{a_0}}\}$ will be the standard basis for $\R^N $ and for 
$\xi \in \A_0$,  we let $e_\xi:= \wedge_{i\in \xi}e_i$. We note that $\{e_\xi\}_{\xi \in \A_0}$ are exactly(up to scalar) the collection of $M$-weight vectors with respect to the $T$-action.

Now finally, when $M=\Q$ and for $B\in G(\R)$ and $\ep>0$, the definition of polytope is:

\begin{defi}
\begin{equation*}
     \begin{aligned}
         \Omega^s_{B,\ep} :&= 
         \{ \, t\in \Lie(T_s(\R))\, \big\vert\,\, ||B(\exp{t}) e_{\xi}||\geq \ep, \; \forall \xi \in \A_0\setminus \{1,...,N\} \}   
    \end{aligned}
\end{equation*}
\end{defi}

To see it is a convex polytope, write $\chi_\xi$ for the weight $\text{Lie}(T(\R)) \to \C$ corresponding to $e_{\xi}$, then
$||B(\exp{t})e_{\xi}||= \exp{(\chi_{\xi}(t))}||Be_{\xi}||
$, hence by taking logarithm on both sides,

\begin{equation}\label{defiPoly1}
    \Omega^s_{B,\ep} = \{ t\in \text{Lie}(T_s(\R)) \,\big\vert \,
    \chi_{\xi}(t)\geq \log(\ep)  - \log{||Be_\xi||}, \; \forall \xi\in \A_0\setminus \{1,...,N\}
    \}
\end{equation}

In the case when $T$ is 
$\Q$-anisotropic, the above definition is trivial as $T_s=\{id\}$. 
And when $T$ is $\Q$-split we retrieve the Definition 4.1 from \cite{ShaZhe18}.

Now one might wish to show for $t\in \Omega^s_{B,\ep}$ and arbitrary $t'\in \text{Lie}(T_a(\R))$, \\
$||B\exp{(t+t')}v||$ is bounded from below 
for \textbf{all} nonzero $v\in\Z^N$. 
This turns out to be not true, however. What we can show, which is sufficient for our purpose, is that it is true for most $t\in \Omega^s_{B,\ep}$. 

Let us fix a bounded nonempty open subset $\Omega_a$ and $\Omega_s$ in $\text{Lie}(T_a(\R))$ and $\text{Lie}(T_s(\R))$ respectively and let $\Omega:=\Omega_a+\Omega_s$. For any natural number $d$, $B\in G(\R)$ and $v\in \bigwedge^d \R^N$, consider the function $\phi_{\Omega,B,v}:3^{N^2}\Omega \to \R$ defined by $\phi_{\Omega,B,v}(t)=||B\exp{(t)}v||$. It is known that there are constants $C(\Omega), \alpha(\Omega)$ depending only on $\Omega$ but not on $B,v$ such that $\phi_{\Omega,B,v}$ is $(C,\alpha)$-good (For a definition, see \cite{Kle10}). At this point it is best to quote Theorem 5.2 from Kleinbock-Margulis \cite{KleMar98}(c.f. \cite{EskMozSha97}, Theorem 3.9). Here it is slightly modified (and downgraded) to suit our situation:

\begin{thm}\label{kleMarg1}
$M=\Q$. Fix $N$ and $T\leq \SL_N(\Q)$ a maximal $\Q$-torus.  Let $\Omega$ be a bounded nonempty open set in $\Lie(T(\R))$.
There exist constants $C,\alpha, \kappa>0$.
For all $B\in \SL_N(\R)=G(\R)$ and $t\in \Lie(T(\R))$, $\ep' \in  (0,1)$ such that for any $d$ and 
nonzero $v\in \bigwedge^d \Z^N $, 
\begin{equation*}
\sup_{t'\in\Omega}|\phi_{\Omega,B\exp{t},v}(t')|>\ep'
\end{equation*}
Then for any $\delta \leq 1$ we have
\begin{equation*}
\frac{1}{|\Omega|} 
|\{t'\in \Omega \,\vert \, \inf_{\lambda\in \Z^N\setminus\{0\}}||(B\exp{(t+t')}\lambda)|| \leq \ep' \delta \}| \leq \kappa \cdot (\delta)^{\alpha}
\end{equation*}
\end{thm}

In order to apply this theorem we need to verify the assumption. That is done by the following proposition.

\begin{prop}\label{nondivgEst1}
$M=\Q$. Fix N and $T\leq \SL_N$ a maximal $\Q$-torus.  Let $\Omega$ be a bounded nonempty open set in $\Lie(T(\R))$. Fix $\ep>0$.  Then there exists a constant $\ep'=\ep'(\ep, \Omega)$. For any 
$B\in G(\R)$, $t\in \Omega_{B,\ep}$, $d \in \{1,...,N-1\}$ and nonzero $v\in \bigwedge^d \Z^N$,
\begin{equation*}
    \sup_{t'\in \Omega} ||B\exp{(t+t')}v||\geq \ep'
\end{equation*}
\end{prop}

In the case of $T_a=T$, there is no "polytopes" and this proposition is a special case of one in \cite{EskMozSha97}, Proposition 4.4. 
In the case of $T_s=T$, this is explained in \cite{ShaZhe18}.
They both use the following consequence of properties of linearly independent functions

\begin{lem}\label{indepFunctLem}
Fix $N$, $T$, $\Omega_a \subset \Lie(T_a(\R))$ and 
$\Omega_s \subset \Lie(T_s(\R))$ as before. 
For $\{\chi_1,...,\chi_r\}$ set of $r$ distinct functionals on $\Omega_a$
there exists a positive constant $\kappa_2=\kappa_2(\Omega,N,\chi)>0$. For any natural number $d$  any $\{v_1,...,v_r\} \subset \bigwedge^d \C^N$, we have 
\begin{equation*}
\sup_{t\in \Omega_a}||\sum_1^r \exp{\chi_i(t)} v_i || \geq \kappa_2 \sup_{i=1,...,r}||v_i||
\end{equation*}
Same statement is true if replacing $\Omega_a$ by $\Omega_s$.
\end{lem}

See for instance Proposition 7.5 of \cite{ShaZhe18} for a proof. We will prove Proposition \ref{nondivgEst1} using Lemma \ref{indepFunctLem} in the next subsection. For now, we note the following corollary of Proposition \ref{nondivgEst1}.

\begin{coro}\label{nondivergence1}
  Fix $N$, $T$, and $\Omega \subset \Lie(T(\R))$ as before. Also fix $\ep>0$.
  For any sequence $g_i$ in $G(\R)$ and $t_i\in \Lie(T(\R))$ 
  satisfying $(t_i+\Omega_s) \cap \Omega^s_{g_k,\ep}\neq \varnothing$, 
   all weak-$*$ limits of $(g_i\exp{t_i})_* (\widehat{\mu_T|_{\exp{\Omega}\Gamma}})$ are probability measures on $G(\R)/\Gamma$.
\end{coro}

\begin{proof}
Indeed, this is simply an interpretation using Mahler's criterion of Theorem \ref{kleMarg1} whose assumption is verified by Proposition \ref{nondivgEst1}. 
\end{proof}

\subsection{Proof of nondivergence}
We don't require $M=\Q$ in this subsection unless explicitly specified. But the reader should feel free to assume $M=\Q$ in a first reading.
\\

Our main result here is the following technical statement. The reader is reminded that we identify $\{1,..., N\}=\{1,...,l_0\}\sqcup \{1,...,l_1\}\sqcup...\sqcup \{1,...,l_{a_0}\}$ and abbreviate those $\{1,...,l_i\}$ as $[l_i]$. 
We also let $[l_i]_{d_i}$ be the collection of subsets of $[l_i]$ of cardinality $d_i$. 
And a Galois number field $L$ in $\overline{\Q}$ that contains all $L_i$'s is fixed .
\begin{prop}\label{lemmanondivergence}
There exists a constant $\kappa_3=\kappa_3(\Omega)>0$. For any $d =1,..., N$ and nonzero
$$\displaystyle{
v=\sum_{ \substack{
I \in [l_0]_{d_0}\sqcup [l_1]_{d_1}\sqcup ...\sqcup [l_{a_0}]_{d_{a_0}}\\ \sum d_i =d
} } 
}
\alpha_I e_I \in \bigwedge^d \OO_M^N ,$$
 there exists
\begin{enumerate}
    \item $(d_1,...,d_{a_0}) \in \{0,1,...,l_1\} \times ... \times \{0,1,...,l_{a_0}\}$;
    \item $J_0\subset [l_0]$ ;\\
(If $v=\sum_{J\in [l_0]_{d_0}} \alpha_J e_J$ then one should omit item 3,4,5. Also one should omit the terms containing $\xi_0$, $\theta$ and let $a_0=0$ in the inequality below.) 
     \item  $\B_0 \subset [l_1]_{d_1}\sqcup ...\sqcup [l_{a_0}]_{d_{a_0}}$ and $\xi_0\in \B_0$;
     \item $\theta \in \Perm \{1,...,a_0\}$ such that $d_{\theta j}/l_{\theta j}-d_{\theta (j+1)}/l_{\theta(j+1)} \geq 0$;
     \item $\sum_{\zeta \in [l_1]_{d_1}\sqcup ...\sqcup [l_{a_0}]_{d_{a_0}}}  \alpha_{J_0\sqcup \zeta}\beta^{\xi_0}_{\zeta} \neq 0$
\end{enumerate}

such that for all $B\in G(\R)$, $t\in \Lie(T_s(\R))$,
\begin{equation*}
    \begin{aligned}
    &\sup_{t'\in \Omega} 
    ||B\exp{(t+t')} v ||
     \geq 
     \kappa_3 
    \big\vert
    \Nm_{L/M}(
    \sum_{\zeta \in [l_1]_{d_1}\sqcup ...\sqcup [l_{a_0}]_{d_{a_0}} } 
     \alpha_{J_0\sqcup \zeta}\beta^{\xi_0}_{\zeta}
    )
    \big\vert
    ^{\frac{1}{|\Gal_0|}} \\
    &\cdot \prod_{j=0}^{a_0} \Big(
    \exp{\big(\sum_{i=1}^j\chi_{[l_{\theta(i)}]}(t)+\chi_{J_0}(t) \big) } 
    || B(e_{J_0}\wedge e_{[l_{\theta 1}]}\wedge...\wedge e_{[l_{\theta j}]})  ||
    \Big)^
    {
    |\B_0|
    (
     \frac{d_{\theta j}}{l_{\theta j}}- \frac{d_{\theta (j+1)}}{l_{\theta(j+1)} } 
    )
    }
    \end{aligned}
\end{equation*}
 $\Gal_0$ is a subgroup of $\Gal(L/M)$ and we set $d_{\theta 0}/l_{\theta 0}:=1$, $d_{\theta (a_0+1)}/l_{\theta (a_0+1)}:=0$. 
 And $\beta^{\xi_0}_{\zeta}$ are some algebraic integers in $\OO_L$ depending on $\xi_0, \zeta$ but nothing else.

\end{prop}

\begin{proof}[Proof of Proposition \ref{nondivgEst1} assuming Proposition \ref{lemmanondivergence}]

Let $M=\Q$. 

$\sum_{\zeta}\alpha_{\zeta}{\beta^{\xi_0}_{\zeta}} \in \OO_L$ is nonzero,
hence $\text{Nm}_{L/\Q}\sum_{\zeta}\alpha_{\zeta}{\beta^{\xi_0}_{\zeta}} \in \Z$ is also nonzero  
and \\
$|\text{Nm}_{L/\Q}(\sum_{\zeta }  \alpha_{\zeta}\beta^{\xi_0}_{\zeta})|^{\frac{1}{|\Gal_0|}}\geq 1$.

On the other hand, $t\in\Omega_{B,\ep}$ implies for all $j$,
\begin{equation*}
    \begin{aligned}
          \ep &\leq ||B\exp{(t)} (e_{J_0}\wedge e_{[l_{\theta 1} ]} \wedge...\wedge e_{[l_{\theta j}]} )|| \\
          & = \exp{\big(\sum_{i=1}^j\chi_{[l_{\theta(i)}]}(t) + \chi_{J_0}(t) \big) } \, || B (e_{J_0}\wedge e_{[l_{\theta 1}]}\wedge...\wedge e_{[l_{\theta j}]}  )  ||
    \end{aligned}
\end{equation*}

Hence by Proposition \ref{lemmanondivergence} above, 
\begin{equation*}
    \begin{aligned}
          \sup_{t'\in \Omega} ||B\exp{(t+t')v }||  
          &\geq 
          \kappa_3 \cdot 1 \cdot \prod_{j=0}^{a_0} \ep^{|\B_0|(d_{\theta j}/l_{\theta j} - d_{\theta(j+1)}/l_{\theta(j+1)}    )} \\
          &= \kappa_3\ep^{|\B_0|}
    \end{aligned}
\end{equation*}

Now, taking $\ep':=\kappa_3 \min\{\ep, \ep^{2^l}\}$ completes the proof.

\end{proof}

Now we turn to the proof of the proposition. 

\begin{proof}[Proof of Proposition \ref{lemmanondivergence}]

Take such a vector $v$ we may rewrite as 
$$v = \sum_{ d_0+d_1+...+d_{a_o}=d} 
\sum_{  \substack{
J \in [l_0]_{d_0}\\
\zeta \in [l_1]_{d_1}\sqcup ...\sqcup [l_{a_0}]_{d_{a_0}}
}
}
\alpha_{J\sqcup \zeta} e_J \wedge e_{\zeta}.$$ 
All coefficients $\{\alpha_{\bullet}\}$ are in $\OO_M$.

Take $t_s\in \text{Lie}(T_s(\R))$, and observe that $\chi_{J\sqcup \zeta}(t_s)$ does not depend on $\zeta$ as long as the tuple $(d_1,...,d_{a_0})$ is fixed. We have 

\begin{equation*}
    \begin{aligned}
         \exp{(t_s)}v &= \sum_{J,d_1,...,d_{a_0}} \sum_{\zeta} \exp(\chi_{J\sqcup \zeta}(t_s)) \alpha_{J\sqcup \zeta} e_{I}\wedge e_{\zeta}\\
         &=\sum_{J,d_1,...,d_{a_0}} \exp(\chi_{J\sqcup \zeta}(t_s)) \sum_{\zeta} \alpha_{J\sqcup \zeta} e_{I}\wedge e_{\zeta}
    \end{aligned}
\end{equation*}

Now the family of characters $\{\chi_{J\sqcup \zeta}\}_{J, d_1,...,d_{a_0}}$ are distinct on $\text{Lie}(T_s(\R))$ hence on $\Omega_s$. By Lemma \ref{indepFunctLem},
 
 \begin{equation} \label{11}
 \sup_{t_s \in \Omega_s} ||B \exp(t+t_a+t_s)v|| \geq 
 \kappa_2 \sup_{|J|+d_1+...+d_{a_0}= d} 
 ||B\exp(t+t_a)\sum_{\zeta}  \alpha_{J\sqcup \zeta} e_J\wedge e_\zeta||
 \end{equation}

Take $d=(d_1,...,d_{a_0})$, $J_0$ that achieves the supreme on the right hand side. 
If $v=\sum_{J\in [l_0]_k} \alpha_J e_J$ then we are already done with $\kappa_3 = \kappa_2$. So assume this is not true.

In order to apply Lemma \ref{indepFunctLem} again to $\sup_{t_a\in \Omega_a}$ we need to "diagonalize" $\exp{t_a}$. Recall that according to our special form of $T$, if we take 
$$A_0:= \text{diag}( I_{l_0}, (\sigma_{\xi}^i v^i_{\zeta})_{(\xi,\zeta)\in [l_1] \times [l_1]},..., (\sigma_{\xi}^i v^i_{\zeta})_{(\xi,\zeta)\in [l_{a_0}] \times [l_1]}  )$$
then $A_0^{-1} \exp{(t)} A_0$ is diagonalized for all $t \in \text{Lie}(T(\R))$ and $A_0$ commutes with $T_s(\R)$. 

For simplicity we abbreviate
$\alpha_{J_0\sqcup \zeta}= \alpha_{\zeta}$, And all the $\beta$'s appearing in the computation below are some minors of $A_0$ or products of them. 
They are algebraic integers in $\OO_L$. 

\begin{equation*}
   \begin{aligned}
        \sum_{\zeta} \alpha_{\zeta} e_{J_0} \wedge e_{\zeta} 
        &= A_0^{-1} \sum_{\zeta} \alpha_{\zeta} A_0 e_{J_0}\wedge e_{\zeta}\\
        &= A_0^{-1}\sum_{\zeta} \alpha_{\zeta} e_{J_0}\wedge (\sum_{\xi_1 \in [l_1]_{d_1}} \beta^{\xi_1}_{\zeta_1}e^1_{\xi_1}  )\wedge ...\wedge (\sum_{\xi_{a_0} \in [l_{a_0}]_{d_{a_0}}} \beta^{\xi_{a_0}}_{\zeta_{a_0}}e^{a_0}_{\xi_{a_0}}  )\\
        &= A_0^{-1}\sum_{\zeta} \alpha_{\zeta} e_{J_0} \wedge \sum_{\xi=\xi_1\sqcup...\sqcup \xi_{a_0}} \beta^{\xi}_{\zeta}e_{\xi}\\     
        &= A_0^{-1}\sum_{\xi} (\sum_{\zeta} \alpha_{\zeta} \beta^{\xi}_{\zeta} ) 
        e_{J_0} \wedge e_{\xi}
   \end{aligned}
\end{equation*}
 
 Let us continue,
 
\begin{equation}\label{13}
   \begin{aligned}
       & \sup_{t_a \in \Omega_a} ||B\exp(t+t_a) \sum_{\zeta} \alpha_{\zeta} e_{J_0} \wedge e_{\zeta}|| \\
       =& \sup_{t_a \in \Omega_a} ||B A_0^{-1} \exp (t + \Ad(A_0)t_a) \sum_{\xi} (\sum_{\zeta} \alpha_{\zeta} \beta^{\xi}_{\zeta} ) e_{J_0} \wedge e_{\xi} || \\
       =&  \sup_{t_a \in \Omega_a} ||B A_0^{-1} \exp (t) \sum_{\xi} (\sum_{\zeta} \alpha_{\zeta} \beta^{\xi}_{\zeta} )  \exp\circ \chi_{J_0\sqcup \xi} (\Ad(A_0)t_a)  e_{J_0} \wedge e_{\xi} ||
   \end{aligned}
\end{equation}

   Now one can verify that the collection of functionals 
   $$\{ t_a \mapsto \chi_{J_0\sqcup \xi}(\Ad(A_0)t_a)  ,\,{\xi \in [l_1]_{d_1}\sqcup...\sqcup [l_{a_0}]_{d_{a_0}}} \}$$
   are distinct on $\Lie(T_a(\R))$ and hence also distinct on $\Omega_a$. Hence 
   
   \begin{equation}\label{13}
   \begin{aligned}
       & \sup_{t_a \in \Omega_a} ||B\exp(t+t_a) \sum_{\zeta} \alpha_{\zeta} e_{J_0}\wedge e_{\zeta}|| \\
       \geq & \kappa_2 
       \sup_{\xi} |(\sum_{\zeta} \alpha_{\zeta} \beta^{\xi}_{\zeta} ) |  ||B A_0^{-1} \exp(t)  e_{J_0} \wedge e_{\xi} || 
   \end{aligned}
\end{equation}

 Take $\xi_0$ such that $\sum_{\zeta} \alpha_{\zeta}\beta^{\xi_0}_{\zeta}\neq 0$. Instead of showing this single term is bounded from below we actually look at the product over its Galois orbit.

 Before we proceed, let us say a few words about the induced actions of Galois group $\Gal=\Gal(\overline{\Q}/M)$. 
 $\Gal$ naturally acts on 
 $\Gal(\overline{\Q}/M)/\Gal(L_i/M)=\{\sigma^i_1,...,\sigma^i_{l_i}\}$ 
 which can be identified with
 $ [l_i]$ under $\sigma_j^i\mapsto j$.
 Hence $\Gal(\overline{\Q}/M)$ acts on $[l_i]$, $[l_i]_{d_i}$ and $[l_1]_{d_1}\sqcup...\sqcup[l_{a_0}]_{d_{a_0}}$. 
 Also, for $\sigma \in \Gal$, $\sigma (\beta^{\xi}_{\zeta})=sgn(\sigma,\xi) \beta^{\sigma(\xi)}_{\zeta}$ 
 for some function $sgn=sgn(\sigma,\xi)$(independent of $\zeta$) taking value in $\{-1,1\}$. 
 Remember $\beta$'s are minors of the matrix $A_0$ and Galois action permutes the rows of $A_0$.

 Now define $\B_0:=\Gal\cdot \xi_0 \subset [l_1]_{d_1}\sqcup...\sqcup[l_{a_0}]_{d_{a_0}}$. And let $\Gal_0$ be the stabilizer of $\xi_0$ in $\Gal(L/M)$ ( Note the $\Gal$ action factors through that of $\Gal(L/M)$ ). Recall $L$ is defined to be a Galois number field that contains all $L_i$'s. So far we've got

\begin{equation}\label{14}
\begin{aligned}
   &\sup_{t'\in \Omega} ||B\exp{(t+t')}v||^{|\B_0|}\\
    \geq &  
    \kappa_2^{|\B_0|} \big(  \sup_{t_a \in \Omega_a} ||B\exp(t+t_a) \sum_{\zeta} \alpha_{\zeta} e_{J_0}\wedge e_{\zeta}|| \, \big) ^{|\B_0|} \\    
\geq & 
\kappa_2^{2|\B_0|} \big\vert 
\prod_{\xi \in\B_0} \sum_{\zeta} \alpha_{\zeta}\beta_{\zeta}^{\xi}
\big\vert
\prod_{\xi \in \B_0} ||B A_0^{-1} \exp{(t)} e_{J_0}\wedge e_{\xi}||\\
=&  \kappa_2^{2|\B_0|} 
\big\vert \prod_{\sigma\in \Gal({L}/\Q)/\Gal_0} 
 \sigma(\sum_{\zeta} \alpha_{\zeta}\beta_{\zeta}^{\xi_0})
 \big\vert 
\prod_{\xi \in \B_0} (\exp{\chi_{J_0\cup \xi}(t)})||BA_0^{-1} e_{J_0}\wedge e_{\xi}||\\
=&  \kappa_2^{2|\B_0|} 
|\text{Nm}_{{L}/\Q}(  \sum_{\zeta} \alpha_{\zeta}\beta_{\zeta}^{\xi_0}  ) |^{1/|\Gal_0|}
 \exp{(\sum_{\xi\in \B_0} \chi_{J_0\cup \xi}(t)) } \prod_{\xi \in \B_0}  ||BA_0^{-1} e_{J_0}\wedge e_{\xi}||
\end{aligned}
\end{equation}

Now it remains to estimate $\exp{(\sum_{\xi\in \B_0} \chi_{J_0\cup \xi}(t)) } \prod_{\xi \in \B_0}  ||BA_0^{-1}  e_{J_0}\wedge e_{\xi}||$. 

Because of the transitivity of Galois action on each $[l_i]$, we can find a non-negative constant $c_i$ for $i=1,...,a_0$ such that 
$c_i=|\{\xi\in\B_0\, \vert\, j^{(i)}\in \xi\}|$ 
for all $j= j^{(i)}\in[l_i]$. 
One can see that 
\begin{equation*}
    \sum_{\xi\in\B_0} \chi_{J_0 \cup\xi } = |\B_0|\chi_{J_0}+\sum_{i=1}^{a_0}c_i\chi_{[l_i]}
\end{equation*}

For each $i$, Applying both sides to $\diag(0,...,0,0 I_{l_{i-1}},I_{l_i},0 I_{l_{i-1}},...0 )$ 
we see $d_i|\B_0|=c_il_i\implies c_i=|\B_0|d_i/l_i$.

Take $\theta \in \Perm\{1,...,a_0\}$ such that $1\geq d_{\theta 1}/l_{\theta 1} \geq d_{\theta 2}/l_{\theta 2} \geq ...\geq d_{\theta a_0}/l_{\theta a_0} \geq 0$.
Also define $d_0/l_0=d_{\theta 0}/l_{\theta 0}=1$ and $d_{a_0+1}/l_{a_0+1}=d_{\theta (a_0+1)}/l_{\theta (a_0+1)}=0$. To lighten the notation we shall omit $\theta$ in the following. 
 
With these notations we have 
\begin{equation*}
\chi_{J_0}+ \sum_{j=1}^{a_0} c_j/|\B_0| \chi_{[l_j]}
=
\chi_{J_0}+ \sum_{j=1}^{a_0} \frac{d_j}{l_j} \chi_{[l_j]} 
= \sum_{j=0}^{a_0} (d_j/l_j - d_{j+1}/l_{j+1}) (\sum_{i=1}^{j} \chi_{[l_i]}+\chi_{J_0})
\end{equation*}
Hence 
\begin{equation}\label{17}
\begin{aligned}
& \exp{(\sum_{\xi\in \B_0} \chi_{J_0\cup \xi}(t)) } \prod_{\xi \in \B_0}  ||BA_0^{-1} e_{J_0}\wedge e_{\xi}|| \\
  =&
  \exp{\bigg( 
  |\B_0|  \sum_{j=0}^{a_0} (d_j/l_j - d_{j+1}/l_{j+1}) (\sum_{i=1}^{j} \chi_{J_0}(t)+ \chi_{[l_i]}(t))  
  \bigg) }
  \prod_{\xi \in \B_0}  ||B A_0^{-1} e_{J_0}\wedge e_{\xi}||  
\end{aligned}
\end{equation}

We take care of the last expression $\prod_{\xi \in \B_0}  ||B A_0^{-1} e_{J_0}\wedge e_{\xi}||  $ now.

For each $\xi=\xi_1\sqcup...\sqcup \xi_{a_0} \in \B_0$ we write $\xi_i=\{k_1^i<...<k_{d_i}^i \}$.  For $I \subset [l_0]\sqcup ...\sqcup [l_{a_0}]$, define
$V_I:= \text{the $\R$-subspace of $\C^N$ represented by }BA_0^{-1}e_I$. Observe that if $I\subset J$, then $V_I \subset V_J$ and hence $\dist(V_I,v)\geq \dist(V_J,v)$ for any vector v (distance is computed with respect to the Euclidean metric).
 For $i\in\{1,...,a_0\}$, 
$k\in [l_i]$, define $k^{<}:=J_0\sqcup[l_1]\sqcup...\sqcup[l_{i-1}]\sqcup\{1,...,k-1\}$.

Now we compute:

\begin{equation}\label{18}
\begin{aligned}
&\prod_{\xi \in\B_0} ||BA_0^{-1} (e_{J_0}\wedge  e_{\xi}) || \\
=& \prod_{\xi\in\B_0}\text{covol} (\oplus_{i\in J_0\cup \xi } \Z BA_0^{-1}e_{i} )\\
= & \prod_{\xi \in\B_0}\text{covol}(\oplus_{i\in J_0} \Z B e_i ) 
\dist(BA_0^{-1}e_{k_1^1} , V_{J_0}) \dist(BA_0^{-1}e_{k_2^1} , V_{J_0\cup\{k_1^1\}}) \cdot... \\
& \cdot \dist(BA_0^{-1} e_{k_{d_{a_0}}^{a_0}} , V_{J_0\cup\{k_1^1,...,k_{d_{a_0}-1}^{a_0}  \}}) \\
\geq & \prod_{\xi\in\B_0}
||B e_{J_0}|| \dist(BA_0^{-1}e_{k_1^1} , V_{(k_1^1)^<}) \dist(BA_0^{-1}e_{k_2^1} , V_{(k_2^1)^<}) \cdot ...\\
& \cdot \dist(BA_0^{-1}e_{k_{d_{a_0}}^{a_0}} , 
    V_{  (k_{d_{a_0}}^{a_0})^{<} })
\\
= & ||B e_{J_0}||^{|\B_0|} \prod_{j=1}^{a_0} \prod_{k\in[l_j]} \dist(BA_0^{-1}e_k, V_{(k)^<})^{|\B_0| d_j/l_j}
\end{aligned}
\end{equation}
Here "covol" is computed from the Euclidean metric.

The right hand side can be transformed into:

\begin{equation}\label{19}
\begin{aligned}
 & ||Be_{J_0}||^{|\B_0|} \prod_{j=1}^{a_0} \prod_{k\in[l_j]} \dist(BA_0^{-1} e_k, V_{(k)^<})^{|\B_0| d_j/l_j}\\
= & ||Be_{J_0}||^{|\B_0|} \prod_{j=1}^{a_0} \prod_{k\in[l_j]} \dist(BA_0^{-1}e_k, V_{(k)^<})^{|\B_0| \sum_{a\geq j}d_a/l_a-d_{a+1}/l_{a+1}} \\
=& \prod_{j=1}^{a_0} ||Be_{J_0}||^{|\B_0|(\frac{d_j}{l_j}- 
\frac{d_{j+1}}{l_{j+1}} )}
  \prod_{k\in[l_1]\sqcup...\sqcup[l_{j}]} \dist(BA_0^{-1}e_k, V_{(k)^<})^{|\B_0|(\frac{d_j}{l_j}- 
\frac{d_{j+1}}{l_{j+1}}) }\\
=& \prod_{j=0}^{a_0} ||Be_{J_0}\wedge BA_0^{-1}e_{[l_1]} \wedge...\wedge BA_0^{-1}e_{[l_j]} ||^{ |\B_0|(\frac{d_j}{l_j}- 
\frac{d_{j+1}}{l_{j+1}}) }\\
=&  \kappa_4 \prod_{j=0}^{a_0} ||Be_{J_0}\wedge Be_{[l_1]} \wedge...\wedge B e_{[l_j]} ||^{ |\B_0|(\frac{d_j}{l_j}/- 
\frac{d_{j+1}}{l_{j+1}}) }
\end{aligned}
\end{equation}
The last equality is because by definition $A_0 e_{[l_i]} = c_i e_{[l_i]}$ for some nonzero constant $c_i \in L$ and we denote by $\kappa_4$ for the absolute value of certain power of the product of $c_i^{-1}$'s.
 
Combining equations  (\ref{14}), (\ref{17}), (\ref{18}) and (\ref{19}) completes the proof.

\end{proof}

\section{Split part, graph and convex polytope}\label{section3}
Take $M=\Q$ in this section.

By Section \ref{section2}, we already have the non-divergence of the sequence of measures $(g_i)_{*} \mu_T\vert_{\exp{\Omega}\Gamma}$. 
One might think we are done with the polytopes. 
However, this is not the case. That the polytopes $\Omega_{g_i,\ep}$ grow "in all directions" under suitable conditions is crucial in showing the limit measure is the $G(\R)$-invariant Haar measure on $G(\R)/\Gamma$(again, under suitable conditions). 
To show this, we adapt the method of Shapira-Zheng \cite{ShaZhe18} to our situation. Indeed, this part only involves the split part of the torus.

The main proposition we establish is the following:

\begin{prop}\label{stablevolume}
Let $\{g_i\}$ be a sequence in $G(\R)$ and assume $\{g_i\}$ diverges in $G(\R)/Z_G(S)(\R)$ for any nontrivial subtorus $S$ of $T_s$.
There exists a sequence of real numbers $\omega_i \to +\infty$. If we define $\widetilde{\Omega}_{g_i,\ep}:=\Omega_{g_i,\ep+\omega_i}$, its volume remains asymptotically the same as $\Omega_{g_i,\ep}$, i.e.
\begin{equation*}
    \lim_{i\to \infty} \frac{\Vol  (\widetilde{\Omega}_{g_i,\ep})}  {\Vol (\Omega_{g_i,\ep})} =1
\end{equation*}
\end{prop}

\subsection{Reduction to unipotents}
First we reduce to the case when $g_i$'s are some special unipotents.

Take $H_{ss}$ to be the subgroup of $G$ consisting of 
$\{\text{diag}(I_{l_0}, A_1,...,A_{a_0})\, \vert \, A_i \in \SL_{l_i} \}$. 
Let $U_T$ be the subgroup of upper triangular matrices that are contracted by 
$\text{diag}(1,2,...,l_0, (l_0+1) I_{l_1},..., (l_{0}+a_0) I_{l_{a_0}})^{-1}$, i.e. matrices of the form: \\

\setcounter{MaxMatrixCols}{20}
$u=
\begin{bmatrix}
  1  &  *  & \cdots& *      & &     & &  &&&  \\
     & 1  & \cdots & *      & &      & & &&&\\
     &    & \ddots & \vdots & & & & \scalebox{4}{$*$}  &&&\\
     &    &      &   1      & &      & &  &&&\\
       &    &    &          & &      & &  &&&\\
       &    &    &          & &   \scalebox{1.5}{$I_{l_1}$}    & & \cdots  && \scalebox{2}{$*$}&\\
       &    &    &          & &      & &  &&&\\
       &    &    &    &&&& \ddots & &\vdots&\\
       &    &    &    &&&&&&&\\
       &    &    &    &&&&&&&\\
        &    &    &    &&&&&&\scalebox{1.5}{$I_{l_a}$}&\\
         &    &    &    &&&&&&&\\
  \end{bmatrix}
$\\

Now $H_{ss}$ commutes with $T_s$, and together with $U_T$ and $T_s$ they form a parabolic subgroup with $U_T$ being the unipotent radical. Hence we may write $g_i=\delta_i u_i h_i \exp{(t_i)}$ 
with $\delta_i$ in $SO(\R)$ and $u_i\in U_T(\R)$, $h_i\in H_{ss}(\R)$, $t_i\in \text{Lie}(T_s(\R))$. Referring to the definition of $\Omega_{B,\ep}$, one sees that $\Omega_{g_i,\ep}= \Omega_{u_ih_i} -t_i =\Omega_{u_i}-t_i$ hence $\text{Vol}{(\Omega_{g_i})}=\text{Vol}(\Omega_{u_i})$. Moreover, as $h_i\exp{(t_i)}$ are in the centralizer of $T_s$, we have $u_i$ diverges in $G(\R)/Z_G(S)(\R)$ for all nontrivial subtorus $S$ of $T_s$. Hence we assume from now on that $g_i=u_i$.

Recall $\A_0:=2^{\{1,...,l_0\}} \times \{\emptyset, [l_1]\} \times ... \times \{\emptyset, [l_{a_0}]\} $ is a collection of subsets of $\{1,...,l_0\}\sqcup [l_1]\sqcup ...\sqcup [l_{a_0}]$.
For $(\xi,\zeta)\in \A_0 \times \A_0 $ and $u \in U_T$ we let $u^{(\xi,\zeta)}$ be the sub-matrix whose rows come from $\xi$ and columns from $\zeta$. Those blocks correspond to direct sums of weight spaces with respect to the action of $\Ad(T_s)$.

\subsection{Graph}
We now want to define a graph associated to this sequence $g_i=u_i$.

To define the graph correctly, we further:
\begin{enumerate}
    \item Divide $\{u_i\}$ into union of finitely many disjoint subsequences and replace $\{u_i\}$ by any one of them.
    \item Modify  $\{u_i\}$ by a bounded sequence in $G(\R)$ from left.
\end{enumerate}

such that for each pair $(\xi,\zeta)\in \A_0\times \A_0$ with $\xi \neq \zeta$, $\{u_i^{(\xi,\zeta)}\}$ diverges or remains $0$. This can always be done using Gauss's elimination.
We note that such a modification is necessary for Proposition \ref{graphConnected} below to hold.

Now we can define a graph $ \mathcal{G}(\{u_i\})=(\mathcal{V},\mathcal{E})$. $\mathcal{V}:=\{1,...,l_0\}\sqcup\{[l_1],...,[l_{a_0}]\}$ and $\{\xi, \zeta\} \subset \mathcal{V} $ is in $\mathcal{E}$ iff $\{u^{(\xi,\zeta)}_i\}$ diverges.

\begin{prop}\label{graphConnected}
Assume $\{u_i\}$ is a sequence in $U_T(\R)$ that satisfies conditions described above, i.e. $\{u_i^{(\xi,\zeta)}\}$ either diverges or remains $0$ for all $\xi \neq \zeta$. Then TFAE:
\begin{enumerate}
\item $\{u_i\}$ diverges in $G(\R)/Z_G(S)(\R)$ for every nontrivial subtorus $S$ of $T_s$.
\item The associated graph $\mathcal{G}(u_i)$ is connected.
\end{enumerate}
\end{prop}

We put a partial order on the set $\mathcal{V}$ by setting $1<2<...<l_0<[l_1]<...<[l_{a_0}]$. Also we use the convention that $l_0+1=[l_1]$, $[l_1]+1=[l_2]$ and so on. 

\begin{proof}
First we assume the graph is connected. This part of the proof is almost identical to that of Lemma 5.2 in \cite{ShaZhe18} which we sketch below. For $s \in T_s(\R)$, assume $x_i:=u_i s u_i^{-1} $ bounded, we want to show $s=id$.

We write $s=\text{diag}(\lambda_1,...,\lambda_{l_0},\lambda_{[l_1]}I_{l_1},...\lambda_{[l_{a_0}]}I_{l_{a_0}}  )$.  By equating the $(\cdot)^{\xi,\zeta}$ term on both sides of $u_i s= x_i u_i$ we get
\begin{equation*}
u_i^{\xi,\xi+1}\lambda_{\xi+1}=\lambda_{\xi} u_i^{\xi,\xi+1}+ \sum_{\xi<k < \zeta} x_i^{\xi,k} u_i^{k,\zeta} + x_i^{\xi,\zeta}
\end{equation*}
Let us fix $\zeta$ for the moment and let $\xi$ vary. We start from $\xi=\zeta-1$, where the  equation above becomes
\begin{equation*}
u_i^{\zeta-1,\zeta}(\lambda_{\zeta}-\lambda_{\zeta-1})= x_i^{\zeta-1,\zeta}
\end{equation*}
Left hand side is either zero or divergent but right hand sides is bounded. Hence we conclude both sides are zero, $x_i^{\zeta-1,\zeta}=0$. Now take $\xi=\zeta-2$,we have 
\begin{equation*}
u_i^{\zeta-2,\zeta-1}(\lambda_{\zeta-1}-\lambda_{\zeta-2})= x_i^{\zeta-2,\zeta}.
\end{equation*}
By the same reasoning, $x_i^{\zeta-2,\zeta}=0$. And continuing to argue this way will show $x_i= s$ which means $s $ commutes with $u_i$. If $s\neq id$, we may find a proper and nonempty subset $I$ of $\mathcal{V}=\{1,...,l_0,[l_1],...,[l_{a_0}]\} $ such that $\lambda_i \neq \lambda_j$ for all $i\in I, j\in I^{c}$. 
In order to commute with such an $s$, $u_i^{\xi,\zeta}$ has to be zero for all $(\xi, \zeta)$ coming from $I \times I^c \cup I^c \times I$. This contradicts the assumption that the graph is connected.

Now let us assume $\mathcal{G}(u_i)$ is not connected(proof is the same as \cite{ShaZhe18}, Proposition 5.5), take a proper subset of vertices $I$ to be a nontrivial connected component. Then it is easy to find $s_{\neq id}= \text{diag}(\lambda_1,...,\lambda_{l_0},\lambda_{[l_1]}I_{l_1},...\lambda_{[l_{a_0}]}I_{l_{a_0}}  )$ in $T_s(\R)$ such that all $\lambda_k$ are the same for $k \in I$ and all $\lambda_j$ are the same for $j\in I^c$. Clearly such an element will commutes with $u_i$ for all $i$.
\end{proof}

Recall the definition of UDS from \cite{ShaZhe18}, Definition 5.6:
\begin{defi}
A subset $J\subset \mathcal{V}$ is called UDS iff for any $j\in J$, any $i\in \mathcal{V}$ that is smaller than $j$, $\{i,j\}\in \mathcal{E}$ implies $i\in J$.
\end{defi}

\begin{prop}\label{graphUDS}
Let $I\subset \mathcal{V}$, $\{u_i\}$ satisfies same assumption as in last proposition. $\{u_i e_I\}$ is bounded iff $I$ is UDS, in which case $u_ie_I=e_I$ for all $i$. Moreover $\{u_ie_I\}$ diverges if $I$ is not UDS.
\end{prop}

\begin{proof}
The proof is really the same as \cite{ShaZhe18}, Proposition 5.7 except that it is messier here.

Assume $I$ is UDS, that means for all pair $(\xi,\zeta)\in I^c \times I$ with $\xi<\zeta$, $u^{(\xi,\zeta)}_i= 0$. Hence for arbitrary $J \neq I$ with $J\in \{1,...,N\}_{|I|}$, we have the $e_J$-coefficients of $u_ie_I$ to be $0$. Hence $u_i e_I =e_I$.

Assume it is not, take  $(\xi,\zeta)\in I^c \times I$ with $\xi<\zeta$, and $u^{(\xi,\zeta)}_i$ diverges. We use $\tilde{\xi}$ to denote the underlying set of $\xi$(e.g. the underlying set of $\xi=\{1,[l_1]\}$ is $\{1,l_0+1,...,l_0+l_1\}$). 

Hence we may find $j\in \tilde{\xi}$ and $k\in \tilde{\zeta}$ such that $(u_i)_{(j,k)}$ diverges and the  coefficients in front of $e_{(\tilde{\zeta}\setminus {k}) \cup \{j\}}$ for 
$u_i e_{\tilde{\zeta}}$ is $(u_i)_{(j,k)}$ which diverges.
\end{proof}

The following lemma, copied from Lemma 5.8 of \cite{ShaZhe18}, is the key thing we require from the graph theory.
\begin{lem}\label{graphLem}
Let $\mathcal{G} (V,E)$ be a connected graph with $V=\{v_1,...,v_n\}$  an ordered set. 
Then we can assign real numbers $x_1,...,x_n$ to the vertices $v_1,...,v_n$ such that
\begin{enumerate}
\item $\sum_{v_i\in V} x_i =0$;
\item For any proper \text{UDS} subset $S\subset V$, $\sum_{v_i\in S}x_i > 0$.
\end{enumerate}
\end{lem}

Note hypothesis of this lemma is satisfied by the $\mathcal{G}(u_i)$ by Proposition \ref{graphConnected}.

\subsection{Back to convex polytope}

The last step towards Proposition \ref{stablevolume} is the following(see also Lemma 6.2 of \cite{ShaZhe18}):

\begin{prop}\label{bigball}
Assume $\{u_i\}$ satisfies the same conditions as in 
Proposition \ref{graphConnected}.
There exists $R_i \to +\infty$ such that $\Omega_{u_i,\ep}$ contains a ball of radius $R_i$.
\end{prop}

\begin{proof}{Proof of Proposition \ref{stablevolume} assuming Proposition \ref{bigball}}
We may assume $g_i=u_i$ and $u_i$ satisfies the conditions as in Proposition \ref{graphConnected}.

Re-centering the convex polytopes, Proposition \ref{bigball} implies $\Omega_{u_i,\ep}$ is a sequence of polytopes with all defining facets diverging to infinity. Thus for any fixed number $\iota>0$, if each side is shrunk by $\iota$, asymptotically the volume remains the same. This is because the difference of their volume is bounded by $\iota \text{Vol}(\partial \Omega_{u_i,\ep})$ which, by Lemma 4.4 from \cite{ShaZhe18}, is bounded by 
$\big( \text{dim}(T_s)/R_i \big) \Vol(\Omega_{u_i,\ep}) $.

As $\iota>0$ is arbitrary, we may now choose $\iota_i >0$ that diverges to $+\infty$ such that after shrinking each facets of $\Omega_{u_i,\ep}$ by $\iota_i$ the volume is still asymptotically unchanged. Referring back to the definition of polytope in the form of equation \ref{defiPoly1}, this means
$\lim_i \text{Vol}(\Omega_{u_i,\ep e^{\iota_i}}) / \text{Vol}(\Omega_{u_i,\ep }) =1$. Now let $\omega_i:= \ep e^{\iota_i}-\ep$ finishes the proof.
\end{proof}

Now we turn to the proof of the last proposition.

Recall by Proposition \ref{graphUDS}, $u_ie_I$ either diverges or remains $e_I$ constantly.
Let $\A_0(bdd):=\{I\in \A_0 \,,\,u_ie_I=e_I \} \setminus \{1,...,N\} $, the set of vertices (except $\{1,...,N\}$) that are UDS(by Lemma \ref{graphUDS}). And 
$\A_0(div):=\{I\in \A_0 \,,\,u_ie_I \,\text{ diverges} \}$. 
Define
\begin{equation}\label{defibddPoly1}
    \Omega^s_{B,\ep}(bdd) = \{ t\in \Lie(T_s(\R)) \,\big\vert \,
    \chi_{\xi}(t)\geq \log(\ep)  - \log{||Be_\xi||}, \; \forall \xi\in \A_0(bdd)
    \}
\end{equation}
Similarly we define $\Omega^s_{B,\ep}(div)$. Then 
$\Omega^s_{B,\ep}=\Omega^s_{B,\ep}(bdd)\cap\Omega^s_{B,\ep}(div) $.
Set $\ep\in (0,1)$ for convenience, then 
$\Omega^s_{B,\ep}(bdd) \supset  \{ t\in \text{Lie}(T_s(\R)) \big\vert \,    \chi_{\xi}(t) > 0\}, \; \forall \xi\in \A_0(bdd)    \}=: \Omega^s_{B,\ep}(bddd) $.
We will show the latter set contains a non-empty open cone which clearly implies Proposition \ref{bigball} as all sides of the other family of polytopes $\Omega^s_{B,\ep}(div)$ already go to infinity by definition.
    
\begin{proof}
It suffices to find a single point in $\Omega^s_{B,\ep}(bddd)$. 

By Lemma \ref{graphLem}, there are $\{x_1,...,x_{l_0}, x_{[l_1]},...,x_{[l_{a_0}]}\}$ such that $\sum_I x_i>0$ for all proper UDS $I$ and $\sum_{\mathcal{V}} x_i=0$. Then the point $t=(t_i)_{i\in\{ 1,..,N\}} \in \text{Lie}(T_s(\R))$ with $t_i=x_i$ for $i=1,..., l_0$ and $t_i=x_{[l_j]}/l_j$ if $i\in[l_j]$ and $j=1,...,a_0$ clearly lies in $\Omega^s_{B,\ep}(bddd)$.
\end{proof}

\section{Limiting measures in the generic case}\label{section4}
We let $M=\Q$ throughout this section.

In this section we shall prove Theorem \ref{theorem0}. 
Let us remind the reader that this means we are given a sequence $\{g_i\} \subset G(\R)=\SL_N(\R)$ diverging in $G(\R)/Z_G(S)(\R)$ for all $S\subset T$ nontrivial $\Q$-subtori. Our goal is to show $\lim_i (g_i)_{*}[\mu_T]=[\mu_G]$. To make things look better, we shall assume $\Gamma \subset G(\Z)$.

\subsection{Reduction to translates of bounded pieces}

\begin{lem}\label{theorem1'}
Given a sequence $\{g_i\} \subset G(\R)$ as above, for any $\ep>0$, we have 
$$\lim_{i \to \infty} \frac{(g_i)_{*}(\mu_T\vert_{\exp{(\Omega^s_{g_i,\ep})}T_a(\R)\Gamma} )}
{{\Vol}( \exp(\Omega^s_{g_i,\ep}) T_a(\R)/\Gamma)} = \mu_G.$$
\end{lem}

\begin{proof}[Proof of Theorem \ref{theorem1} assuming Lemma \ref{theorem1'}]

For any $f\in C_c(G(\R)/\Gamma)$ take $\ep>0$ small enough such that the support of $f$ is contained in 
$$X_{\ep}:=\{g\Gamma \,\vert\, \inf_{v\in \Z^N,\,v\neq 0}||gv|| \geq \ep\}.$$
Note if $t\in \text{Lie}(T_s(\R))$, $t_a\in \text{Lie}(T_a(\R)) $  and $g_i\exp{(t+t_a)}\Gamma \in X_{\ep}$ 
then for all $I\in \A_0$,
$||g_i\exp{(t+t_a)}e_I||=||g_i\exp{(t)}e_I||\geq \ep $ and hence 
$t\in \Omega^s_{g_i,\ep}$.

Hence we have
$$\frac{( f, (g_i)_{*}\mu_T )}{{\Vol}( \exp(\Omega^s_{g_i,\ep}) T_a(\R)/\Gamma)}
=\frac{( f , (g_i)_{*} (\mu_T\vert_{\exp{(\Omega^s_{g_i,\ep})}}))}
{{\Vol}( \exp(\Omega^s_{g_i,\ep}) T_a(\R)/\Gamma)}
 $$
with the latter converging to $( f , \mu_{G})$. Hence we are done.
\end{proof}

By assumption, $\{g_i\}$ diverges in $G(\R)/Z_GS(\R)$ for all nontrivial $S\subset T_s$ since any subtorus of a $\Q$-split torus is automatically defined over $\Q$. Hence Proposition \ref{stablevolume} applies. Thus the above lemma follows from

\begin{lem}\label{theorem1''}
Given a sequence $\{g_i\} \subset G(\R)$ as above, for any $\ep>0$, we have 
$$\lim_{i \to \infty} \frac{(g_i)_{*}(\mu_T\vert_{\exp{(\widetilde{\Omega}^s_{g_i,\ep})}T_a(\R)\Gamma} )}
{\Vol(\exp(\widetilde{\Omega}^s_{g_i,\ep})T_a(\R)/T_a(\R)\Gamma)} = \mu_G.$$
\end{lem}

This is still not translates of a fixed piece yet. 

\begin{thm}\label{theorem1'''}
$\{g_i\} \subset G(\R)$ same as above.
Fix bounded nonempty open subsets  $\Omega_a$ $(\Omega_s)$ of $\Lie(T_a(\R))$($\Lie(T_s(\R))$ resp.) and let $\Omega:=\Omega_a + \Omega_s$. Fix $\ep >0$. 
 
Take $t_i \in \Lie(T_s(\R)) $ s.t. 
$t_i + \Omega_s \bigcap \widetilde{\Omega}^s_{g_i,\ep} \neq \emptyset$. 
Then 
$$\lim_{i \to \infty}(g_i\exp{t_i})_* (\widehat{\mu_T|_{\exp{\Omega}\Gamma}})=\mu_G$$
\end{thm}

\begin{proof}[Proof of Lemma \ref{theorem1''} assuming Theorem \ref{theorem1'''}]
The lemma follows by decomposition $\Omega^s_{g_i,\ep}T_a(\R)\Gamma/\Gamma$ into smaller pieces. Let us take $\Omega_a$ whose image under $\exp$ is a fundamental domain for $T_a(\R)$ with respect to $T_a(\R)\cap \Gamma$. Take $\Omega_s$ to be a fixed cube in $\text{Lie}(T_s(\R))$ and for each $i$, take a tiling, disjoint union of translates of $\Omega_s$, $\Omega_i:=\bigsqcup_{j\in \mathcal{I}_i}t_j+ \Omega_s$ that covers $\Omega^s_{g_i,\ep}$ and each cube intersects $\widetilde{\Omega}^s_{g_i,\ep}$ nontrivially. Then the left hand side in the equation in Lemma \ref{theorem1''} is unchanged if  $\widetilde{\Omega}^s_{g_i,\ep}$  is replaced by $\Omega_i$.

Fix $f \in C_c(G(\R)/\Gamma)$. Let 
$a_{ij}:=(f , \widehat{ \mu_T\vert_{\exp{(t_j+\Omega_s)+\Omega_a} } }    )$ 
for $j \in \mathcal{I}_i$. 
Then by Theorem \ref{theorem1'''}, for all $j_i\in \mathcal{I}_i$, 
$\lim_i a_{i,j_i}$ = $(f,\mu_G)$. 
Hence
$$\frac{(f, \mu_{T}\vert_{\exp(\Omega_i)T_a(\R)})}
{\Vol(\exp(\Omega_i)T_a(\R)/\Gamma)} 
= \frac{\sum_{j\in \mathcal{I}_i}a_{i,j}}{|\mathcal{I}_i|}$$
whose limit is equal to $(f,\mu_G)$.
\end{proof}

The remaining part of this section devotes to a proof of Theorem \ref{theorem1'''}

\subsection{Equidistribution}

By Corollary \ref{nondivergence1}, we already know the sequence of measures $(g_i\exp{t_i})_* (\widehat{\mu_T|_{\exp{\Omega}\Gamma}})$ is non-divergent. Hence there exist a bounded sequence $\delta_i$ in $G(\R)$, and $\gamma_i\in\Gamma$ such taht $g_i\exp{t_i}=\delta_i \gamma_i$. Thus it suffices to show $(\gamma_i)_* (\widehat{\mu_T|_{\exp{\Omega}\Gamma}}) \to \mu_G$.  

We start by quoting Theorem 2.1 from \cite{EskMozSha96} with some modifications
\begin{thm}\label{Theom2.1}
Let G be a connected real algebraic groups defined over $\Q$ and $H$ a reductive $\Q$-subgroup. Let $\Gamma\subset G(\R)$ be an arithmetic lattice and
$\Lambda:= \Gamma\cap H(\R)$. 
Let $\rho_i: H \to G $ a sequence of $\Q$-homomorphisms with the following properties:
\begin{enumerate}
\item No proper $\Q$-subgroup of G contains $\rho_i(H)$ for infinitely many $i\in \N$.
\item For every $h\in H(\Q)$, there exists $k\in \N$ such that $\{\rho_i(h)\}_{i\in \N} \subset G(\frac{1}{k}\Z)$.
\item For any sequence $h_i \to e$ in $H(\R)$, all the eigenvalues for the action of $\Ad(\rho_i(h_i))$ on the Lie algebra of $G$ tend to $1$ as $i \to \infty$.
\item For any regular algebraic function $f$ on $G$, the functions $f\circ \rho_i$ span a finite dimensional space of functions on $H$.
\item For all $i\in\N$. $\rho_i(\Lambda) \subset \Gamma$.
\end{enumerate}
Fix a bounded nonempty open subset $O$ of $H(\R)$ and take $\mu_i$ to be the normalization of the restriction of the  $\rho_i(H(\R))$-invariant measure to $\rho_i(O)\Gamma$. Then $\mu_i \to \mu_G$ as $i \to \infty$.
\end{thm}

There are only two modifications that we made in the statement. One is that we restrict the $\rho_i(H(\R))-$invariant measure to a bounded piece.
Indeed, this is how the original proof goes. 
The second is that we replace the assumption $H(\Z)\subset H(\R)$ is a lattice by requiring $H$ to be reductive. Indeed, the fact that $H(\Z)\subset H(\R)$ is a lattice is never used in the proof after we restrict to a bounded piece. See \cite{EskMozSha96} and \cite{EskMozSha96Errat} for details.

We now set $H=T$ and $\rho_i$ defined 
by $t \mapsto \gamma_i t \gamma_i^{-1}$. All assumptions $2\sim 5$ above can be verified. So the proof of Thoerem \ref{theorem1'''} boils down to verifying the first one:

\begin{prop}
Assume the sequence $\{g_i\} $ diverges in $G(\R)/Z_GS(\R)$ for all nontrivial $\Q$-subtori $S\leq T$ and $t_i + \Omega_s \bigcap \widetilde{\Omega}^s_{g_i,\ep} \neq \emptyset$.
 Assume $\{\gamma_i\} \subset G(\Z)$ satisfies
 $g_i\exp{t_i}=\delta_i \gamma_i$ with $\delta_i$ bounded. 
 Then the $\Q$-Zariski closure of the group generated by $\bigcup_i \gamma_i T \gamma_i^{-1}$ is $G$.
\end{prop}

\begin{proof}
It is known that any maximal subgroup of $G$ defined over $\Q$ is either parabolic or reductive. Hence it suffices to show it is impossible to find groups of such types that contains $\bigcup_i \gamma_i T \gamma_i^{-1}$. In the first case, contained in a parabolic implies there exist a flag, which we may assume to be minimum, $\{\{0\}< V<\Q^N\}$ that is stablized simultaneously by $\bigcup_i \gamma_i T \gamma_i^{-1}$. In the second case, contained in a reductive group containing a maximal torus implies there exist a nontrivial $\Q$-torus that  centralize $\bigcup_i \gamma_i T \gamma_i^{-1}$. Then our proof is complete assuming Proposition \ref{parabolicQ} and \ref{reductiveQ} below.
\end{proof}

Before we goes into the proof of these two propositions, let us translate the properties enjoyed by $g_i$ and $t_i$ to $\gamma_i$. 

First, that $g_i $ diverges in $G(\R)/Z_GS(\R)$ for all nontrivial $\Q$-subtori $S\leq T$, 
becomes: 
\begin{enumerate}
     \item[] $\{\gamma_i\} $ diverges in $G(\R)/Z_GS(\R)$ for all nontrivial $\Q$-subtori $S\leq T$.
\end{enumerate}
 
Second, that $t_i + \Omega_s \bigcap \widetilde{\Omega}^s_{g_i,\ep} \neq \emptyset$ implies: 
\begin{enumerate}
   \item[] For all $I\in \A_0 \setminus{\{1,...,N\}}$, $||\gamma_ie_I||=||\delta_i^{-1}g_i\exp(t_i) e_I|| \geq \omega_i$ for some $\omega_i \to +\infty$.
\end{enumerate}


\begin{prop}\label{parabolicQ}
  Any nontrivial proper $\Q$-subspace $V$ of $\Q^N$ can not be stablized by $\bigcup_i \gamma_i T \gamma_i^{-1}$.
  \end{prop} 
  
  \begin{proof}
  Assume not true and take such a subspace. Let $e_V$ be an integer vector in $\wedge^{\text{dim} V}\Z^N$ that represents $V$. $\gamma_iT(\R)\gamma_i^{-1}e_V \subset \R e_V$ implies $\gamma_ie_V$ are rational eigenvectors for $T$-action. Hence there exist $I_i\in\A_0 \setminus \{1,...,N\}$ such that $\gamma_ie_V=\lambda_i e_{I_i}$. Passing to a subsequence we may assume $I_i=I$ for all $i$. Now $\gamma_i^{-1}$ preserves integral points and hence $\lambda_i \in \Z$, implying $|\lambda_i| \geq 1$. This in turn implies that $||e_V||=|\lambda_i|||\gamma_i e_I||$. This is a contradiction because the right hand side is unbounded as $i$ varies.
  \end{proof}
  
  It remains to take care of the reductive case. We need the following lemma.
  
  \begin{lem}\label{reductiveLem1}
  Let $T$ be a maximal torus of $G$ and $S$ is a subtorus, the set
  $Z(S,T)(\C)$  $:=\{g\in G(\C)\, , \, gS(\C)g^{-1}\in T(\C)\}$ 
  is equal to 
  $N_G(T)(\C)Z_G(S)(\C)$. 
  \end{lem}
  
  \begin{proof}
  $N_G(T)(\C)N_G(S)(\C) \subset Z(S,T)(\C)$ is easy, we only need to show the other inclusion. So take $A\in Z(S,T)(\C)$. As we are working over an algebraically closed field, we may assume $T$ consists of diagonal matrices. Take $s \in S(\C)$ generic in the sense that $Z_G(S)=Z_G(s)$. By assumption, $AsA^{-1}$ and $s$ are both diagonalized and hence we may take $w\in N_G(T)(\C)$ a permutation matrix such that $w^{-1}AsA^{-1}w=s$. Therefore $w^{-1}A \in Z_G(s)(\C)=Z_G(S)(\C)$ and $A\in N_G(T)(\C)$.
  \end{proof}
  
  \begin{lem}\label{reductiveLem2}
  Suppose a sequence $\{g_i\}$ in $G(\R)$ diverges in $G(\R)/Z_G(S)(\R)$, then it also diverges in $G(\C)/Z_G(S)(\C)$.
  \end{lem}
  
  \begin{proof}
  It suffices to show for $v_{\neq 0}\in \text{Lie}(S(\C))$, 
  $ \Ad(g_i)v $ diverges. 
  Indeed such a $v$ is a semisimple element and can be written as 
  $v=v_1+iv_2$ with $v_i ,v_2 \in \Lie(S(\R))$ and semisimple. 
  Hence $ \Ad(g_i)(v_1+iv_2)=  \Ad(g_i )v_1 + i \Ad(g_i)v_2 $ 
  diverges and we are done.
  \end{proof}
  
  \begin{prop}\label{reductiveQ}
  Take $T\leq G$ a maximal $\Q$-torus, $\gamma_i$ same as above. Then $\bigcup_i \gamma_i T \gamma_i^{-1}$ can not centralize any nontrivial $\Q$-torus $S$.
  \end{prop}
  
  \begin{proof}
  Assume the proposition is false and such an $S$ exists. Let $S':= \gamma_1^{-1}S\gamma_1$. Then $\gamma_1^{-1}S\gamma_1 \subset T$ and is centralized by and hence contained in $\gamma_1^{-1} \gamma_i T \gamma_i^{-1} \gamma_1 $ for all $i$. In other words conjugation by $\gamma_i^{-1}\gamma_1$ takes $S'$ into $T$. 
  
  By Lemma \ref{reductiveLem1}, $\gamma_i^{-1}\gamma_1 =w_i z_i$ for some 
  $w_i \in N_G(T)(\C)$ and $z_i \in Z_G(S')(\C)$. But $N_G(T)(\C)$ is finite modulo $T(\C) \subset Z_G(S')(\C)$, we may therefore assume $w_i$ comes from a fixed finite set and by passing to a subsequence, we may further assume $w_i=w_0$ constantly.

  Now $\gamma_i=\gamma_1 z_i^{-1}w_0^{-1} = \gamma_1 w_0^{-1} (w_0 z_i w_0^{-1}) $. 
  And $w_0 z_i w_0^{-1} \in Z_G(w_0^{-1}S'w_0)(\C)$ hence $\gamma_i$ remains constant in $G(\C)/ Z_G(w_0^{-1}S'w_0)(\C) $. 
  By definition of $w_0$, $w_0^{-1}S'w_0$ is a subtorus of $T$.
  It is also defined over $\Q$ because $w_0 = z_i \gamma_1^{-1} \gamma_i$ and $w_0^{-1}S'w_0 =  \gamma_i^{-1} \gamma_1 S' \gamma_1^{-1} \gamma_i $. By Lemma \ref{reductiveLem2} this contradicts our assumption that $\gamma_i$ is supposed to diverge in such a quotient.

  \end{proof}

\section{Restriction of scalars}\label{section5}

In this section prove Theorem \ref{theorem2}. 
We don't require $M=\Q$ anymore. $G= \SL_N$ is viewed as an algebraic group defined over $M$ and we let $G':=R_{M/\Q}G$. Similarly, $T':=R_{M/\Q}T$. It is known that $T'$ is a maximal torus of $G'$ defined over $\Q$ and every maximal torus of $G'$ defined over $\Q$ arises this way 
(see for instance the appendix of \cite{ConGabPra15}).
Let $\Gamma' \subset G'(\R)$ be a lattice commensurable with $G'(\Z)=G(\OO_M)$.

Take a sequence $\{g_i\} \subset G'(\R)$ that diverges in $G'(\R)/Z_{G'}S'(\R)$ for all nontrivial $\Q$-subtorus $S'$ of $T'$ . 
We need to show 
$\lim_i(g_i)_*[\mu_{T'}] = [\mu_{G'}]$. 

The basic structure of the proof remains the same as Section \ref{section2}, \ref{section3} and \ref{section4}. For this reason, at most places we only indicate the necessary modifications on the statement and omit the proof. However, the proof of the non-divergence requires some extra argument and we shall be more careful. 

\subsection{Preparations}
We want to collect some notations and lemmas on restriction of scalars here. See \cite{Spr98} or \cite{PlaRap94} for a reference.

Take $\{w_1,...,w_{m_0}\}$ to be a basis for $\OO_M$ as a free $\Z-$Module. 
Let $\{\tau_1,...,\tau_{m_0}\}$ be a full set of representatives in $\Gal(\overline{\Q}/\Q)$ of $\Gal(\overline{\Q}/\Q)/\Gal(\overline{\Q}/M)$.
We let $\{\alpha_1,...,\alpha_{r_0}\}$ be the set of real embeddings of $M$ 
and $\{\beta_1,...,\beta_{s_0}\}\sqcup \{\overline{\beta_1},...,\overline{\beta_{s_0}}\}$ 
be conjugate pairs of complex embeddings of $M$. 
So $r_0+2s_0 = m_0$ and
we may identify $\alpha_i$'s, $\beta_i$'s with $\{\tau_1,..,\tau_{m_0}\}$.

Take $V$ to be $M^N$ we have been working with. $V':=\Res_{M/\Q}V$ is defined to be $V'(\R)=V(\R)^{\oplus r_0}\bigoplus V(\C)^{\oplus s_0}$ as a ($m_0 N$-dimensional) real vector space with an integral($\Z$) structure given by the following basis:
\begin{equation*}
    \big\{     (\alpha_1(w_k e_j),...,\beta_{s_0}(w_k e_j)   )\, \big\vert \, 
    k\in \{1,...,m_0\},
    j\in \{1,...,N\}
    \big\}
\end{equation*}

Actually, $V'$ is not so different from $V$ in that $V(M)\cong V'(\Q)$, $V(\OO_M)\cong V'(\Z)$ where the bijection is given by $v \mapsto v':=(\alpha_1v,...,\beta_{s_0}v)$. And $V'(\Q)$ inherits an $M$-module structure via this bijection, $w\cdot (v_1,...,v_{m_0})= (\alpha_1(w)v_1,...,\beta_{s_0}(w)v_{r_0+s_0})$ for $w\in M$ and $(v_1,...,v_{m_0})\in V'(\Q)$. There is a crucial difference, however, between $\bigwedge^aV$ and $\bigwedge^a V'$. For instance, when $m_0 > 1$ and $a=N+1$, the former space is reduced to $\{0\}$ while the latter is not.

For $v=v_1\wedge...\wedge v_a \in \bigwedge^a V(M)$, define a "norm" map $N_{M/\Q}: \bigwedge^a V(M) \to \bigwedge^{am_0} V'(\Q) $ (it also sends $\bigwedge^a V(\OO_M)$ to $\bigwedge^{am_0} V'(\Z)$) by


\begin{equation*}
\begin{aligned}
    &N_{M/\Q}(v_i)\\
    :=& w_1 \cdot v'_i \wedge w_2 \cdot v'_i \wedge ...\wedge w_{m_0}\cdot v'_i \\
    =&  (\alpha_1(w_1v_i),...,\beta_{s_0}(w_1v_i))\wedge ...\wedge (\alpha_1(w_{m_0}v_i),...,\beta_{s_0}(w_{m_0}v_i))\\
    =& \Big( \alpha_1(w_1)(\alpha_1v_i,0,...,0)+...+ \alpha_{r_0}(w_1)(0,...,\alpha_{r_0}v_i,0...)+ \Re(\beta_1 w_1)(0,...,\beta_1v_i,0,...)+ \\
    & \Im(\beta_1 w_1)(0,...,i\beta_1v_i,0,...)+...+ \Im(\beta_{s_0} w_1)(0,...,0,i\beta_{s_0}v_i) \Big) \wedge....\wedge \\
    &\Big( \alpha_1(w_{m_0})(\alpha_1v_i,0,...,0)+...
    +  \Im(\beta_{s_0} w_{m_0})(0,...,0,i\beta_{s_0}v_i)\Big)  \\
    =&  \frac{1}{2i}\det{(\tau_k(w_j))_{k,j}} (\alpha_1v_i,0,...,0) \wedge...\wedge (0,...,0,i\beta_{s_0}v_i)
\end{aligned}
\end{equation*}
For the last equality we used the fact that for two complex number $z,w$, 
$2i \cdot \text{det}\begin{psmallmatrix} \Re z & \Im z\\ \Re w & \Im w\end{psmallmatrix} = \text{det}\begin{psmallmatrix} z & \overline{z}\\ w & \overline{w}\end{psmallmatrix} $.

\begin{equation*}
\begin{aligned}
    &N_{M/\Q}(v)\\
    :=& N_{M/\Q}(v_1)\wedge...\wedge N_{M/\Q}(v_{a})\\
    =& (-1)^{a(a-1)m_0(m_0-1)/4} (\frac{1}{2i}\det{(\tau_k(w_j))_{k,j}})^a (\alpha_1v,0,...,0)\wedge...\wedge(0,...,0,\beta_1 v,0,...)\wedge \\
    &(0,...,0,i\beta_1 v ,0,...)\wedge...\wedge (0,...,0,i\beta_{s_0}v)
\end{aligned}
\end{equation*}

We now explain the motivation to define this $N_{M/\Q}$ function. for $v_1,...,v_a$ in $V(M)$ independent over $M$, we take the $\OO_M$-module spanned by them and regard it as a $\Z$-module sitting inside $V'(\Z)$, then it has a set of basis given by $\{w_i\cdot v'_j\}_{i,j}$ whose exterior product gives $N_{M/\Q}(v)$ up to sign.

We equip $V'(\R)=V(\R)^{\oplus r_0}\bigoplus V(\C)^{\oplus s_0}$ with the product Euclidean metric under which 
$||(v_1,...,v_{r_0+s_0})||^2= \sum ||v_i||^2 $ . 
Under this metric, we can talk about covolume of a $\Z$-submodule in the $\R$-span of this lattice. Taking exterior products of a set of (real) orthonormal basis induces metrics on all exterior powers of $V'$.
We have if $\Lambda'= \Z v'_1 \oplus...\oplus \Z v'_{a}$ then $\text{covol}(\Lambda')=||v'_1\wedge...\wedge v'_{a}||$. Also note that 
\begin{equation*}
    ||N_{M/\Q}v||= |\frac{1}{2}\det{(\tau_k(w_j))_{k,j}} |^a \prod_{i=1,..,m_0}||\tau_iv||
\end{equation*}

\begin{lem}\label{decreseVol}[Decrease of covolume]
There is a constant $\kappa_5>0$ depending only on the field $M$ and dimension $N$ of the vector space such that for all $v\in V(\OO_M)$, 
\begin{equation*}
    ||N_{M/\Q}v||\leq \kappa_4 ||v'||^{m_0}
\end{equation*}
\end{lem}

\begin{proof}
Indeed,
\begin{equation*}
    \begin{aligned}
         \text{LHS}^2 &= |\frac{1}{2}\det{(\tau_k(w_j))_{k,j}} |^2\prod_{i}||\tau_iv||^2 \\
         &\leq |\frac{1}{2}\det{(\tau_k(w_j))_{k,j}} |^2 (\sum_i ||\tau_i v||^2 /m_0)^{m_0}= \text{RHS}^2
    \end{aligned}
\end{equation*}
  where $\kappa_4>0$ is defined by the last equation.
\end{proof}

Define $\PPP_0:=\{\text{ Primitive } \Z \text{-sub-lattices in } V'(\Z)\,\}$, $\PPP_{M}:=\{\Lambda \in \PPP_0 ,\, \OO_M\Lambda \subset \Lambda \}$.
To apply the theorem of Kleinbock-Margulis we need information on all primitive lattices from $\PPP_0$. But at first we will only be able to gain information about $\PPP_M$. The following sequence of lemmas is to build relations between them.

\begin{lem} \label{lem7}
Given $\Lambda \in \PPP_0$ a primitive $\Z$-module in $V'(\Z)$. If there is a full rank $\Z$-submodule $\Lambda'\subset \Lambda$ that is preserved under $\OO_M$-action, then so is $\Lambda$, i.e. $\Lambda\in \PPP_{M}$.
\end{lem}

\begin{proof}
Take $\Lambda'\subset \Lambda$ as in the lemma, then we can find a natural number R s.t. $\frac{1}{R} \Lambda \subset \Lambda'$. 

$\OO_M \Lambda'\subset \Lambda'$ implies 
$\OO_M 
(\frac{1}{R} \Lambda) \bigotimes_{\Z}\Q  
= \OO_M \Lambda \bigotimes_{\Z}\Q \subset \Lambda'\bigotimes_{\Z}\Q$. 
But $\Lambda$, being primitive, is equal to $(\Lambda'\bigotimes_{\Z}\Q) \cap V'(\Z)$ and $\OO_M$ preserves $V'(\Z)$. So we are done.
\end{proof}

\begin{lem}\label{freemoduleVectCorrep}
Given $v=v_1\wedge...\wedge v_a \in \bigwedge^a V(\OO_M)$.  The $\Z$-module generated by $\{w_k\cdot v'_i\}_{k,i}$, the one corresponding to $N_{M/\Q}v$, is also a free $\OO_M$-module with basis $v'_1,...,v'_a$. Conversely all free $\OO_M$-modules in $V'(\Z)$ arise this way.
\end{lem}

\begin{proof}
The first assertion is because $\{w_k\cdot v'_i\}_{k,i}$ is $\Q$-independent. Conversely, take a free $\OO_M$-module of rank $a$ with basis $v'_1,...,v'_a$, then it comes from $v:=v_1\wedge...\wedge v_a$.
\end{proof}

\begin{lem}\label{classnumber}
There is a constant $\kappa_5>0$  depending only on the field $M$ and dimension $N$ of $V$ such that any $\OO_M$-submodule of $V'(\Z)$ has a free $\OO_M$-submodule of index smaller than $\kappa_5$.
\end{lem}

\begin{proof}
Take $\Lambda$ to be such an $\OO_M$-submodule of rank $l+1$. As $\OO_M$ is a Dedekind domain, by Theorem 1.32 in \cite{Nar04} and that our $\OO_M$-module is torsion-free, we see $\Lambda\cong \OO_M^{\oplus l} \bigoplus I$ as an $\OO_M$-module for some ideal $I \subset \OO_M$. Hence it suffices to find a copy of $\OO_M$ inside $I $ of index bounded by some $\kappa_5>0$.

Recall the class group of a number field is finite, hence we can find a finite set of ideals $\{I_1,...,I_k\}$ such that any other ideal maps bijectively to one of them by multiplying by some element in M. Clearly this bijection would be an $\OO_M$-module isomorphism. So to prove the lemma, it suffices to prove it for a single $I$ allowing the  index of a free submodule to depend on $I$. We can find $l \leq $ class number such that $I^l \cong \OO_M$ by finiteness of class number. So $I^k\subset I$ is a free module of index $|I/I^k|=|R/I^{k-1}|=|R/I|^{k-1}$ which is finite. So we are done. 
\end{proof}

Just as $G$ acts on $V$ we have $G'$ acts on $V'$. This gives a model of $G'$ to work with. For $g\in G(M)$ we let $g'\in G'(\Q)$ be
$(\alpha_1(g),...,\alpha_{r_0}(g),\beta_{1}(g),...,\beta_{s_0}(g) ) \in \SL_N(\R)^{\oplus r_0} \times \SL_N(\C)^{\oplus s_0}$.
The following compatibility can be checked directly:

\begin{lem}\label{lem10}
For $g\in G(M)$ and $v\in \bigwedge^a V(M)$, 
\begin{equation*}
    g' N_{M/\Q}v = N_{M/\Q}(gv)
\end{equation*}
\end{lem}

\subsection{Polytope of non-divergence}

Similar to the definition \ref{defiPoly1} from Section \ref{section2},  we can define a polytope lying inside $\text{Lie}(T'_s(\R))$. Up to $\Q$-conjugacy of $T'$, we may assume the split part of the torus $T$ is defined over $\Q$. Note in this case $T'_s(\R)$ consists of $(t,t,...,t)$ with $t\in T_s(\R)$.
For $B=(\alpha_1 B,...,\alpha_{s_0}B) \in G'(\Q)$ and $\ep >0$, we define

\begin{equation*}
\begin{aligned}
    \Omega'_{B,\ep} 
    :=& \big\{ (t,...,t) \in \text{Lie}(T'_s(\R)) \,\Big\vert\,
    ||B\Delta(\exp{t}) N_{M/\Q}e_{\xi}|| \geq \ep ,\; \forall \xi \in \A_0\setminus
    \{1,...,N\} 
    \big\}\\
    =& \big\{(t,..,t)\in \text{Lie}(T'_s(\R))  \Big\vert m_0\chi_I(t)\geq \log{\frac{\ep}{|\det(\tau_iw_j)|^{|\xi|}} }  -
    \log{\prod_{i=1,...,m_0}||\tau_i Be_\xi||} ,\\
    & \quad \forall \xi \in \A_0\setminus
    \{1,...,N\} 
    \big\}
\end{aligned}
\end{equation*}

This polytope can be defined for all $B=(B_1,...,B_{r_0+s_0}) \in G'(\R)$. One simply replace $\alpha_i B$ above by $B_i$ and $\beta_i B$ by $B_{r_0+i}$.

The action of $G'$ on $V'$ gives $G'$ a representation defined over $\Q$ and by passing to a finite index subgroup we may assume $\Gamma'$ is contained in $G'(\Z)$ and preserves $V'(\Z)$. Then we have a proper map from $G'(\R)/\Gamma'$ to $\SL(V'(\R))/\SL(V'(\Z))$. To show nondivergence in $G'(\R)/\Gamma'$ then reduces to show nondivergence in $\SL(V'(\R))/\SL(V'(\Z))$ where we can apply the theorem of Kleinbock-Margulis. However we will not be able to verify the assumptions from Theorem \ref{kleMarg1}. So instead we use Theorem 4.1 from \cite{KleMar98}(c.f. Theorem 3.9 in \cite{Kle10}). Some definitions need to be made before we state it (in a form that suits our situation better). 

For $\Omega\subset \text{Lie}(T'(\R))$ nonempty open bounded, $B\in G'(\R)$ and $v\in \bigwedge^{a'}V'(\R)$, let $\phi_{\Omega,B,v}$ be the function from $3^{m_0N^2}\Omega$ to $\R$ defined by $t \mapsto ||B\exp{(t)}v||$. It is known that such a function is $(C,\alpha)-$good for some $C, \alpha$ depending on $\Omega$ but not on $B$ or $v$. 

Recall $\PPP_M$ denotes the collection of primitive lattices of $V'(\Z)$ that are invariant under $\OO_M$. For a primitive lattice $\Lambda$, we also let $\Lambda$ to denote the exterior product of a $\Z$-basis of $\Lambda$. This is well-defined up to sign. 

Fix $0<\rho<1$. For $\delta>0$, $B$, $\Omega$ as above, we say a point $t\in \Omega$ is $\delta$-protected relative to $\PPP_M$ if there exists a flag 
$\mathcal{F}\subset \PPP_M$ such that (see Section 3.3 of \cite{Kle10})
\begin{enumerate}
\item $\delta \rho^{rk(\Delta)} \leq \phi_{\Omega,B,\Delta }(t) \leq \rho^{rk(\Delta)}$ for all $\Delta \in \PPP_M$.
\item $\phi_{\Omega,B,\Delta } \geq \rho^{rk(\Delta)}$ for all $\Delta\in \PPP_M\setminus \mathcal{F}$ that is comparable to every elements of $\mathcal{F}$.
\end{enumerate}
If we replace $\PPP_M$ by the larger set $\PPP_0$, then this property would imply $B \exp{(t)} \Gamma'$ is trapped in some compact set of $G'(\R)/\Gamma'$(Proposition 3.8 of \cite{Kle10}).

\begin{thm}\label{kleMarg2}
Fix $N$ a natural number, a number field $M$ and a maximal $\Q$-torus $T'$ in $G'$. Fix nonempty bounded open $\Omega\subset \Lie(T'(\R))$. Then there exist constants $\kappa_7>0$ and $\alpha>0$ such that for all $B\in G'(\R)$, all $\rho \in (0,1)$. If we assume for any 
$\Delta\in \PPP_M$,
$\sup_{t\in \Omega}||\phi_{\Omega,B,\Delta}(t)|| \geq \rho$, 
then for any $\delta\in (0,1)$
\begin{equation*}
\frac{1}{|\Omega|} 
|
\{
t\in \Omega\,,\, \text{t is not $\delta$-protected relative to $\PPP_M$}
\}
|
\leq \kappa_7 \delta^{\alpha} 
\end{equation*}
\end{thm}

The assumption is verified by the following proposition for $\rho= \min\{0.5, \ep'/\kappa_6\}$
as all $\Lambda\in\PPP_M$ can be written as a bounded multiple(bounded by $\kappa_6$) of $N_{M/\Q}v$ for some $v\in \bigwedge^a(\OO_M)$ thanks to Lemma \ref{freemoduleVectCorrep} and Lemma \ref{classnumber}. 

\begin{prop}\label{nondivRestriclemma}
Take $\ep>0$ and $\Omega'$ bounded nonempty open set in $\Lie(T'(\R))$.
  There exists $\ep'=\ep'(\ep,\Omega')>0$. For all $B\in G'(\R)$, $(t,...,t) \in \Omega'_{B,\ep}$, $v_{\neq 0}=v_1\wedge ...\wedge v_a$ with $v_i\in V(O_M)$, we have 
   \begin{equation*}
     \sup_{t'\in\Omega'} ||B \Delta(\exp{t})\exp{(t')} N_{M/\Q}v|| \geq \ep'
 \end{equation*}
\end{prop}

Proof is similar to that of Proposition \ref{nondivgEst1} and we omit it.

Now we need to understand what $\delta$-protected could bring to us. 

\begin{lem}\label{protected}
Fix $\Omega, B, \rho\in (0,1)$ as above. Let us abbreviate $\phi_{\Omega,B,\Lambda}(t)$ as $\phi(\Lambda)$.
Assume $t\in \Omega$ is $\delta$-protected for some $\delta>0$. Then 
$\phi(v')\geq \frac{\rho \min\{1,\delta\}^{1/m_0}}{\kappa_5^{1/m_0}}$ 
for all $v_{\neq 0}\in V(\OO_M)$.
\end{lem}
\begin{proof}
Take $v_{\neq 0}\in V'(\Z)$. By approximation we may assume $B\exp{(t)} \in G'(\Q)$.
First we claim that $\phi(N_{M/\Q}v)\geq \min\{1,\delta\}\rho^{m_0}$. Given the claim, recall Lemma \ref{decreseVol}, we have 
$\kappa_5 ||B\exp{(t)}v'||^{m_0} 
\geq ||N_{M/\Q}(B\exp{(t)}v)||=||B\exp{(t)}(N_{M/\Q}v)|| 
\geq  \min\{1,\delta\}\rho^{m_0}$, 
implying 
$$\phi(v') \geq \frac{\rho \min\{1,\delta\}^{1/m_0}}{\kappa_5^{1/m_0}}$$.

The proof of the claim is the same as that of the proof of Proposition 3.8 in \cite{Kle10} which we recall here.

$N_{M/\Q}v$ represents a $m_0$-dimensional $\Z$-module stable under $\OO_M$-action. Take the unique primitive lattice $\Lambda$ that contains it. Note $||\phi(N_{M/\Q}v)||\geq ||\phi(\Lambda)|| $.

By definition of $\delta$-protected we can find a flag, which we may assume contains $\{0\}$ and $V'(\Z)$, 
$\mathcal{F}:=\{\Delta_0=\{0\} \subset \Delta_1 \subset...\subset \Delta_{r+1}=V'(\Z)  \}$. If $\Lambda \in \mathcal{F}$, then $\phi(\Lambda) \geq \delta\rho^{m_0}$. 
So now we assume $\Lambda \notin \mathcal{F}$ and take $i$ such that $v' \in \Delta_{i+1}\setminus \Delta_i$.

Consider $\Lambda':= (\Lambda_{\Q}\oplus (\Delta_i)_{\Q}) \bigcap V'(\Z)$. 
It is of 
rk $\Lambda'=$ rk $\Delta_i +m_0$ and 
invariant under $\OO_M$ because the finite index sublattice 
$\Lambda \oplus \Delta_i$ is 
and we have Lemma \ref{lem7}. Also, $\Lambda'$ is comparable to all $\mathcal{F}$. Therefore
\begin{equation}
\begin{aligned}
     \rho^{\text{rk}(\Delta_i)+m_0} \leq \phi(\Lambda') \leq \phi(\Lambda \oplus \Delta_i) \leq \phi(\Lambda ) \phi(\Delta_i)
     \leq  \phi(\Lambda ) \rho^{\text{rk}\Delta_i}
\end{aligned}
\end{equation}
This implies $\phi(\Lambda) \geq \rho^{m_0}$ and we are done.
\end{proof}

In view of this lemma we may apply Theorem \ref{kleMarg2} together with Mahler's criterion to see
\begin{coro} \label{nondivg2}
Fix $N$, $T'$ and $M$ as before. Let 
$\Omega_s \subset \Lie(T'_s(\R)) $  
and $\Omega_a \subset \Lie(T'_a(\R))$ 
be nonempty open bounded subsets. 
Define 
$\Omega:= \Omega_a + \Omega_s$. Fix also $\ep >0$.  
Then for all sequence $\{g_i\}\subset G'(\R)$ and 
$t_i \in \Lie(T'(\R))$ with 
$t_i + \Omega_s \cap \Omega'_{g_i,\ep}\neq \emptyset$, 
all weak-$*$ limits of 
$(g_i\exp{t_i})_* (\widehat{\mu_{T'}|_{\exp{\Omega}\Gamma'}})$
are probability measures on $G'(\R)/\Gamma'$.
\end{coro}

\subsection{Graph and convex polytope}

Similar to Proposition \ref{stablevolume}  we have
\begin{prop}\label{stablevolume2}
Assume $g_i$ diverges in $G'(\R)/Z_{G'}(S')(\R)$ for all $S'\subset T'_s$ subtorus.
There exists a sequence of real numbers $\omega'_i \to +\infty$ such that if we define $\widetilde{\Omega}'_{g_i,\ep}:=\Omega'_{g_i,\ep+\omega'_i}$, its volume remains asymptotically the same as $\Omega'_{g_i,\ep}$, i.e.
\begin{equation*}
    \lim_{i\to +\infty} \frac{\Vol  (\widetilde{\Omega}'_{g_i,\ep})}  
    {\Vol (\Omega'_{g_i,\ep})} =1
\end{equation*}
\end{prop}

The proof follows the same argument as that from Section \ref{section3}. By approximation we may assume we are dealing with elements from $G'(\Q)= G(M)$. We first reduce to the case when $g_i=u_i$ are some unipotents. In the definition of graph, we replace $u_i^{(\xi,\zeta)}$ by $\prod_{\xi} \tau_{\xi}u_k^{(i,j)}$. And replace all $e_I$ by $N_{M/\Q}e_I$. The rest are the same.

\subsection{Limiting measures}

Now we are to prove Theorem \ref{theorem2}. With the same argument as before we first reduce to translates of a fixed piece.

\begin{thm}
Let $G'=R_{M/\Q}SL_N$ for $M$ a  number field and $T'$ be a maximal $\Q$-torus in $G'$. $\Gamma'\subset G'(\Q)$ an arithmetic lattice. 
Given a sequence $\{g_i\} \subset G'(\R)$ that diverges in 
$G'(\R)/Z_{G'}(S')(\R)$ for all nontrivial $\Q$-subtori $S' \subset T'$. 
Fix $\ep>0$ and a nonempty bounded open subset 
$\Omega$ of $\Lie(T'(\R))$. 
Take $t_i \in \Lie(T'(\R))$ s.t. $t_i + \Omega \cap \widetilde{\Omega}^s_{g_i,\ep} \neq \emptyset$. Then 
$$\lim_{i\to \infty} (g_i\exp{t_i})_* (\widehat{\mu_{T'}|_{\exp{(\Omega)}\Gamma'}}) = \mu_{G'}$$
\end{thm}

Then using Theorem \ref{Theom2.1} as before the above theorem is a corollary of the following:

\begin{prop}
Same assumptions as above.

Then we can find $\gamma_i \in G'(\Z)$ and $\delta_i\in G'(\R)$ bounded with $\delta_i \gamma_i =g_i \exp{(t_i)}$ and the $\Q$-Zariski closure of the group generated by $\bigcup_i \gamma_iT'\gamma_i^{-1}$ is $G'$.
\end{prop}

The proof of this proposition is still the same as before, i.e. we want to eliminate any possible intermediate reductive and parabolic groups. The connection is that a maximal reductive or parabolic subgroup of $G'$ always comes from restriction of scalar of some maximal reductive or parabolic subgroup of $G$. Hence we omit the details.

\section{Intermediate cases}\label{section6}

In this section we deduce Theorem \ref{theorem1} from Theorem \ref{theorem2}. 
The key thing here is to understand all possible $S_0$ and $H_0:=Z_G(S_0)$ that appears in the definition of $(S_0,T)$-clean. 

In the $\Q$-split case this is easy: $S_0$ can be defined by any partition of $\{1,...,N\}$ and $H_0 \cong $ products of $\SL_{\bullet}$'s and a $\Q$-split torus. 
Hence one may use Theorem \ref{theorem0} applied to this product. 
In the $\Q$-anisotropic case, for us this means $l_0=0$ and $a_0=1$, 
the work of Tomanov \cite{Tom00}, Proposition 3.3 and 3.5, 
shows $S_0$ corresponds to some subfield $K$ of $L_1$ and $H_0 \cong R_{K/\Q}\SL_{\bullet}$. 
We extend this analysis to all other cases and therefore this section will have some overlap with \cite{Tom00} which we do not assume. 

\begin{prop}\label{intermediate1}
Assume $M=\Q$ and there exists a sequence $\{g_i\}$ in $G(\R)$ that is $(S_0,T)$-clean. Then there are number fields $K_1,...,K_r$ and natural numbers $k_1,...,k_r$ with $H_0\cong (S_0)_s \times R_{K_1/\Q}\SL_{k_1} \times R_{K_r/\Q}\SL_{k_r} $. 
\end{prop}

\begin{proof}[Proof of Theorem \ref{theorem1} assuming Proposition \ref{intermediate1}]
$T\subset H_0$ is a maximal torus, we can write $T=(S_0)_s \times T_1 \times T_2 \times...\times T_{r}$ with each $T_i$ being maximal torus of $R_{K_i/\Q}\SL_{k_i} $. Applying Theorem \ref{theorem2} then concludes the proof.
\end{proof}

The rest of this section is devoted to a proof of the above proposition.

\subsection{A reduction}

In this subsection make the reduction to the case when $l_0=0$ and $S_0=(S_0)_a$.

\begin{lem}
Assume there are sequences that is $(S_0,T)$-clean, then $Z(H_0)=S_0$.
\end{lem}

\begin{proof}
Let $Z(H_0)=S$ then $Z_G(S)=H_0 =Z_G(S_0)$ hence $g_i \in Z_G(S)$ with $S$ a larger torus also defined over $\Q$. This forces $S=S_0$.
\end{proof}

\begin{lem}
Assume there are sequences that is $(S_0,T)$-clean and let $(S_0)_s$ be the $\Q$-split part of $S_0$. 
Then $ZZ_G((S_0)_s)=(S_0)_s$.
\end{lem}

\begin{proof}
Indeed $ZZ_G((S_0)_s)\subset T$ and hence 
$ZZ_G((S_0)_s)\subset ZZ_G(S_0)=S_0$ and 
$ZZ_G((S_0)_s)\subset ZZ_G(T_s)=T_s$ so 
$ZZ_G((S_0)_s)\subset S_0\cap T_s = (S_0)_s$. The converse inclusion is obvious so we are done.
\end{proof}

Let $H_s$ be the semisimple part of $Z_G((S_0)_s)$, then $(S_0)_a \subset H_s$
and $H_0= Z_{H_s}(S_a) \cdot (S_0)_s$. 
To continue, we need to review the correspondence between partitions and subtorus.

Let $D$ be the maximal torus of $G$ consisting of diagonal matrices. For a partition $\xi=\{\xi_1,...,\xi_k\} $ of 
$\{1,...,l_0,[l_1],...,[l_{a_0}]\}$, 
we let $D_{\xi}\subset T_s$ be the subtorus whose points are of the form 
$$\{ A= \text{diag}(t_1,...,t_{l_0}, t_{[l_1]}I_{l_1},...,t_{[l_{a_0}]}I_{l_{a_0}} ) 
\, \vert\, \det{A}=1 , \,t_i=t_j ,\,\forall i\sim j \}$$ 
where $i\sim j$ iff $i,j$ belong to the same part of this partition. Similarly, for a partitions of $\{1,...,N\}$ we use the same notations. And for partitions of $[l_i]$ we use $D^i_{\xi}$ for the corresponding subtorus.

Recall $
A_0=\text{diag}(  
I_{l_0} ,
(\sigma^1_{i}v^1_j)_{(i,j)\in[l_1]\times [l_1]},...,
(\sigma^{a_0}_{i}v^{a_0}_j)_{(i,j)\in[l_{a_0}]\times [l_{a_0}]}  )$ and $A_0 T {A_0}^{-1} =D$ and $A_0$ commutes with $T_s$. Hence by assumption, there exist partition 
$\zeta_0$ of $\{1,...,N\}$ s.t. $A_0 S_0 A_0^{-1} =D_{\zeta_0}$ 
and partition $\xi_0$ of 
$\{1,...,l_0,[l_1],...,[l_{a_0}]\}$ s.t. $A_0 (S_0)_s A_0^{-1} =D_{\xi_0}$. 

Let us also recall that the collection of $A_0tA_0^{-1}$ for $t\in T(\Q)$ consist of elements of the form
$t_s \cdot \text{diag}(I_{l_0}, 
\sigma^{1}_1(x^1), \sigma^1_2(x^1),....,
\sigma^1_{l_1}(x^1) , \sigma^2_{1}(x^2),....,\sigma^{a_0}_{l_{a_0}} (x^{l_{a_0}})  ) $
 with $x^i\in L_i$ of norm $1$ and 
 $t_s \in T_s(\Q)$ is equal to \
 $\text{diag}(  y_1,...,y_{l_0}, y_{[l_1]} I_{l_1} ,..., y_{[l_{a_0}]} I_{l_{a_0}}
 )$ with $y_i > 0$.
 We shall abbreviate them as $y(t) \cdot x(t)$.

\begin{lem} \label{6.4}
Take $k\in\{1,...,a_0\}$ and $i\in \{1,...,l_0\}$. TFAE: 
\begin{enumerate}
\item There exist $j\in [l_k]$ s.t. $i\sim_{\zeta_0} j$.
\item For all $j\in [l_k]$ we have $i\sim_{\zeta_0} j$.
\end{enumerate}
Moreover, when the above holds we have $i \sim_{\xi_0} [l_k]$.
\end{lem}

\begin{proof}
We need to show $(1)$ implies $(2)$. $i\sim_{\zeta_0} j$ implies for all $t \in S_0$, $\sigma^k_j(x^{j}) y_{[l_k]} = y_{i}$. In particular $x^{j}$ is a rational number and so $\sigma^k_j(x^{j})$ does not depend on the choice of $j$ and they are all equal. So $(2)$ holds.

Now assume this is true. $\sigma^k_j(x^{j}) y_{[l_k]} = y_{i}$ becomes
$ \pm 1 y_{[l_k]} = y_{i}$ ($x^j$ is of norm $1$) hence $y_{[l_k]} = y_{i}$ because they are both positive. So we are done.
\end{proof}

Now if $\xi_0 =\{\xi_1,...,\xi_{|\xi|}\}$, then $H_s= \prod H_{\xi_i}$ with $H_{\xi_i} \cong \SL_{|\xi_i|}$. Take 
$H'_s := \prod_{\xi_i\cap\{1,...,l_0\} =\emptyset} H_{\xi_i}$.

\begin{prop}\label{reduction1}
For $i=1,...,|\xi|$ let $\pi_i$ be the projection from $H_s$ to $H_{\xi_i}$. Set $\pi_i((S_0)_a)=S_i$. Then 
$S_i=\{id\}$ when 
$\xi_i\cap\{1,...,l_0\} \neq \emptyset$, 
$Z(Z_{H_{\xi_i}}S_i)=S_i$ when
$\xi_i\cap\{1,...,l_0\} =\emptyset$ and 
$$(S_0)_a= \prod_{\xi_i\cap\{1,...,l_0\} =\emptyset} S_i\,,
\quad 
Z_{H_s'}(S_0)_a = \prod_{\xi_i\cap\{1,...,l_0\} =\emptyset} Z_{H_{\xi_i}}S_i
.$$
As a result,
$$H_0= (S_0)_s \cdot
\prod_{\xi_i\cap\{1,...,l_0\} \neq \emptyset} H_{\xi_i} \cdot
\prod_{\xi_i\cap\{1,...,l_0\} =\emptyset} Z_{H_{\xi_i}}S_i .$$
\end{prop}

It is clear that in light of this proposition, we may assume $l_0=0$ and $S_0=(S_0)_a$ in the proof of Proposition \ref{intermediate1}.

\begin{proof}
$S_i=\{id\}$ when 
$\xi_i\cap\{1,...,l_0\} \neq \emptyset$ follows immediately from Lemma \ref{6.4}. Then we have $(S_0)_a \subset \prod_{\xi_i\cap\{1,...,l_0\} =\emptyset} S_i$. Let us prove the converse inclusion.

Take $(g_i)_i \in Z_{H'_s}(S_0)_a$, then each component $i$, $g_i$ commutes with $S_i$ hence $(g_i)_i$ commutes with $\prod_i S_i$. This implies $ \prod_{\xi_i\cap\{1,...,l_0\} =\emptyset} S_i \subset Z(Z_{H'_s}(S_0)_a)$.  
But by definition of $H_s$, $H'_s$ and $Z_G(S_0)=S_0$, we see that $Z(Z_{H'_s}(S_0)_a)=(S_0)_a$. So we are done. The verification of the rest of the statement is pretty straight-forward.
\end{proof}

\subsection{Proof of the proposition}
We work under the assumption that  $l_0=0$ and $S_0=(S_0)_a$ in this subsection. The plan is, roughly speaking, to find a number field $K$ with embeddings into all $L_1,...,L_{a_0}$. Taking the diagonal embedding then gives $S_0$. We start by addressing "all".

\begin{lem}\label{connectedParti}
Recall $\zeta_0$ is the partition corresponding to $A_0S_0A_0^{-1}$. Then for all 
$i,j \in \{1,...,a_0\}$ there exists $s\in [l_i]$ and $t\in [l_j]$ with $s\sim_{\zeta_0}t$.
\end{lem}

\begin{proof}

If not, then we can divide $\{[l_1],...,[l_{a_0}]\}$ into two parts $I_1,I_2$ and elements $t=\text{diag}(t_1,...,t_N)$ such that $t_i\neq t_j$ for all $i\in I_1$, $j\in I_2$. Thus $ZZ_G(\{t\})$ contains a $\Q$-split torus. However $ZZ_G(S_0)$ contains $ZZ_G(\{t\})$ and we arrive at a contradiction. 
\end{proof}

Then we starts to construct $K$ and its embeddings. 

For $i=1,...,a_0$, fix $b_i\in [l_i]$ such that $\{b_1,...,b_{a_0}\}$ belongs to the same partition and let $\mathcal{I}_i$ be $\{j\in [l_i] \,,\, j\sim b_i\}$.

Define $$K':=\{\sigma^{i}_{b_i}(k_i) \,\vert\, k_i \in L_i 
\text{ such that }\sigma^i_{j}(k_i) =\sigma^i_{b_i}(k_i) ,\,\forall j\in \mathcal{I}_i\}$$
It will be seen that this definition does not depend on the choice of $i$. One may take $i=1$ in the definition above for now.

Recall $j \mapsto [\sigma^i_j]$ gives a bijection between $\{1,...,l_i\} $ and 
$\Gal(\overline{\Q}/\Q)/\Gal(\overline{\Q}/L_i)$. $\Gal:=\Gal(\overline{\Q}/\Q)$ naturally acts on the latter set and hence on $[l_i]$ for each $i$. So it also acts on $\{1,...,N\}=[l_0]\sqcup ...\sqcup [l_{a_0}]$.
Assume the partition $\zeta_0=\{\zeta_1,...,\zeta_n\}$ is such that $\{b_i\} \subset \zeta_1$.
So $\mathcal{I}_i$ is equal to $\zeta_1 \cap [l_i]$.

\begin{lem}
For $\sigma \in \Gal(\overline{\Q}/\Q)$, TFAE:
\begin{enumerate}
\item $\sigma$ preserves $\zeta_1$.
\item $\sigma$ preserves $\mathcal{I}_i$ for all $i$.
\item $\sigma$ preserves $\mathcal{I}_i$ for some $i$
\end{enumerate}
\end{lem}

\begin{proof}
By definition $\sigma$ preserves each $[l_i]$ and hence (1),(2) are equivalent. Also (2) clearly implies (3).

But one observes that $x\sim y$ iff $\sigma x \sim \sigma y$. 
Then by Lemma \ref{connectedParti}, (3) also implies (2).
\end{proof}

The collection of $\sigma$ that satisfies the above lemma is a finite index subgroup and hence there is a finite extension $K$ of $\Q$ such that it is equal to $\Gal(\overline{\Q}/K)$. This field coincides with $K'$ constructed above.

\begin{prop}
$K=K'$. As a corollary, 
$$A_0 S_0(\Q) A_0^{-1} = \{ 
\diag (
\sigma^1_1 (\sigma^1_{b_1})^{-1}(k)    
, \sigma^1_2 (\sigma^1_{b_1})^{-1}(k),
...,
\sigma^{a_0}_{l_{a_0}} (\sigma^{a_0}_{b_{a_0}})^{-1}(k)
) \,\vert \, k\in K
\}.$$
\end{prop}

\begin{proof}
We may assume $i=1$ in the definition of $K'$.

For $\sigma \in \Gal(\overline{\Q}/K)$ and $k=\sigma^1_{b_1}k_1\in K'$, $\sigma \sigma^1_{b_1}k_1 =\sigma^1_{j}k_1 $ for some $j\in \mathcal{I}_1$. But by definition $\sigma^1_{j}k_1 = \sigma^1_{b_1}k_1$. So $K'$ is fixed by $\Gal(\overline{Q}/K)$ and $K'\subset K$.

The converse is more delicate. First, because $A_0S_0A_0^{-1}=D_{\zeta_0}$, by denseness of rational points in $T_a$, we may find $x_1 \in L_1$ with 
$\sigma^1_{b_1}x_1=\sigma^1_{j}x_1$ for all $j\in \mathcal{I}_1$ but 
$\sigma^1_{b_1}x_1\neq \sigma^1_{j}x_1$ for all $j\notin \mathcal{I}_1$. By construction, $x:= \sigma^1_{b_1}x_1$ is in $K'$.

Now if $K'$ is a proper subfield of $K$ we would find
$\sigma \in \Gal(\overline{\Q}/K') \setminus \Gal(\overline{\Q}/K)$. 
This means (1) $\sigma$ fix $K'$ and 
(2) $\sigma \sigma^1_{b_1} = \sigma^1_{j} $(mod $\Gal(\overline{\Q}/L_1)$) for some $j\notin \mathcal{I}_1$.
By properties of $x$ above, this (1) implies  
$\sigma \sigma^1_{b_1}x_1= \sigma x=   x = \sigma^1_{b_1}x_1$ but (2) implies  $\sigma \sigma^1_{b_1}x_1 
\neq \sigma^1_{b_1}x_1$. This is a contradiction, so $K'=K$.
\end{proof}

It only remains to describe $H_0=Z_G(S_0)$ now.

\begin{proof}[Proof of Proposition \ref{intermediate1}]

Fix $\{\iota_1,...,\iota_r\}$ to be a set of embeddings of $K$ into $\overline{\Q}$ over $\Q$. For each $i$, arrange the order of $\sigma^i_1,...,\sigma^i_{l_i}$ such that the first $r$ elements coincide with $\iota_1\circ \sigma^i_{b_i},...,\iota_r\circ \sigma^i_{b_i} $ and so does the next $r$ elements and so on. Then fix a basis of $K$ over $\Q$, $u_1,...,u_r$. For each $i$, take 
$\{w^i_1,w^i_2,...,w^i_{l_i/r}\}$ to be a basis of $L_i$ over $(\sigma^i_{b_i})^{-1}K=:K_i$. Then we may take 
$v^i_1= (\sigma^i_{b_i})^{-1}(u_1) w^i_1, \,
v^i_2 = (\sigma^i_{b_i})^{-1}(u_2) w^i_1 ,\,
..., 
v^i_{l_i/r+1}=(\sigma^i_{b_i})^{-1}(u_1) w^i_2
$ and so on. 
The point of arranging them into this order is two-fold.

First if we write for $k\in K$, 
$M(k):=\text{diag}(\iota_1 k,...,\iota_r k)$. 
Then 
$A_0S_0(\Q)A_0^{-1}=\{\text{diag}(M(k),M(k),...,M(k)) \,\vert\, k\in K  \}$. If we decompose $A$, an $N$-by-$N$ matrix, into $r$-by-$r$ block matrices then we can see  
its centralizer are those with diagonal block matrices. We use $F$ to denote this group of matrices of coeffients in $\overline{\Q}$.

Second for each $i$, we define $A^{(i)}$ and $B^{(i)}$ by decomposing into $r$-by-$r$ block matrices. We have 
$(A^{a,b})_{j,k}=
\text{diag}(\sigma^i_{(a-1)r+1}w_b , 
\sigma^i_{(a-1)r+2}w_b,...,
\sigma^i_{ar}w_b
)$
 and 
$(B^{a,b})_{j,k} = \iota_j u_k
$ if $a=b$ and is zero matrix otherwise.

Then one can compute $(\sigma^i_jv_k)_{(j,k)} = A^{(i)} \cdot B^{(i)}$. 
Write $A=\text{diag}(A^{(1)}, ..., A^{(a_0)} )$ and  $B=\text{diag}(B^{(1)}, ..., B^{(a_0)} )$. So $A_0= A\cdot B$.
Then for all $s\in A_0S_0(\Q)A_0^{-1}$, 
$A_0^{-1}s A_0 = B^{-1}sB$, 
implying 
$ B^{-1} \{\text{diag}(M(k),M(k),...,M(k)) \,\vert\, k\in K  \} B$ 
 is defined over $\Q$ and that $B^{-1}FB \subset H_0(\overline{\Q})$.

Then we can compute that $B^{-1}fB \in H_0({\Q})$ iff each $r$-by-$r$ block of $f$ is $M(k)$ for some $k$. As $R_{K/\Q}\SL_{N/r}(\Q)$ can also be conjugate over $\overline{\Q}$ into such matrices(see 2.1.2 of \cite{PlaRap94}), we have an algebraic isomorphism between $H_0$ and $R_{K/\Q}\SL_{N/r}$ carrying $H_0(\Q)$ to $R_{K/\Q}\SL_{N/r}(\Q)$. As rational points are Zariski dense in the underlying variety of a connected $\Q$-linear algebraic group, such a morphism has to be defined over $\Q$. So we are done.
 
\end{proof}

\section{Comparison with related results and examples}\label{section7'}

In this section we compare our results with previously obtained ones and give proofs of examples in the introduction. These examples are not covered by previously known results.

\subsection{Comparison with related results}

The case in Theorem \ref{theorem0} or \ref{theorem2},  and when $T$ (or $T'$) is 
$\Q$-anisotropic follows from Corollary 1.13 of \cite{EskMozSha96}. Indeed, all reductive subgroups containing $T$ (or $T'$) is of the form $Z_G(S)$ for some subtorus $S$ of $T$ (or $T'$). But \cite{EskMozSha96} deals with much more general situations than just maximal $\Q$-anisotropic torus in $\SL_N$(or $R_{M/\Q}\SL_N$).

When $T$ is $\Q$-split we recover Theorem 2.4 of \cite{ShaZhe18}. Let us briefly explain why. 

Let $D$ be the full diagonal torus of $G$.
Recall the Definition 2.2 in  \cite{ShaZhe18} where $S$ is a subtorus of $D$ defined over $\R$:
\begin{equation*}
    \mathcal{A}(S(\R),g_i) = \{ 
    Y \in \Lie(S(\R))\, \vert \, \{\Ad(g_i)Y\} \text{ is bounded in } \Lie(G(\R))
    \}
\end{equation*}

Take $x \in G(\R)$ such that $D(\R)x\Gamma/\Gamma$ divergent in $G(\R)/\Gamma$. 
According to a theorem of Tomanov and Weiss (\cite{TomWei03}, Theorem 1.1), $T:=x^{-1}Dx$ is a $\Q$-split torus. And $(g_i)_* \mu_{D(\R)x\Gamma} = (g_i x )_* \mu_{T}$.

Without loss of generality let us assume $\{g_i x\}_{i\in \N} \subset G(\R)$ is $(S_0, T)$-clean and for any $Y \in \Lie(D(\R))\setminus \mathcal{A}(D(\R),g_i)$, $\{\Ad(g_i)Y\}$ diverges. Let $H_0=Z_G(S_0)$ and $H_1=Z_G(\mathcal{A}(D(\R),g_i))$.

Then Theorem \ref{theorem1} above asserts $(g_i)_* [\mu_{Dx\Gamma}] \to [\mu_{H_0}] $ 
while Theorem 2.4 from \cite{ShaZhe18} asserts $(g_i)_* [\mu_{D(\R)x\Gamma}]$ converges to some translates of $[\mu_{H_1(\R) x \Gamma}]$. 
Indeed,  
$x^{-1}H_1x = Z_G(\Ad(x^{-1})\mathcal{A}(D(\R),g_i))= Z_G (\mathcal{A}(T,g_ix))$. 
And 
$\Ad(g_ix)Y $ 
diverges for all $Y\in \Lie(T(\R))\setminus \mathcal{A}(T(\R),g_ix)$.  

On the other hand, $\{g_i x\}_{i\in \N} \subset G(\R)$ is $(S_0, T)$-clean, 
hence $g_ix y (g_ix)^{-1}$ diverges for all $y \in T(\R) \setminus S_0(\R) $ 
and so 
$\Ad(g_ix)Y$ diverges for all
$Y \in $ $\Lie(T(\R))\setminus \Lie(S_0(\R))$. 
Also $\Ad(g_ix)Y = Y$ 
for all $Y \in  \Lie(S_0(\R))$.
These facts show $\mathcal{A}(T(\R),g_ix)$
is no other than $\Lie(S_0(\R))$ 
and thus its centralizer is $H_0$. 
In conclusion we have $x^{-1}H_1x =H_0$, and $\mu_{H_0}$ is a translate of $\mu_{H_1(\R) x\Gamma}$. So our explanation completes.\\

Additionally one could see from Theorem 2.4 from \cite{ShaZhe18} that whether the limit measure is $G(\R)$-invariant is independent of the choice of base point $x$ if $D(\R)x\Gamma$ is divergent. But this could be false if $D(\R)x\Gamma$ is merely closed.
Although the condition posed in \cite{ShaZhe18}, Theorem 2.4 is certainly sufficient for the convergence to the $G(\R)$-invariant Haar measure.

\subsection{Examples}

\paragraph{Example 1} 
Let $N=3$ and $p$ be a non-square integer. 
let
$$  T = \bigg\{ A= \begin{bmatrix}
  b &    c  &            \\
     pc    & b &    \\
         &       & a   \\
  \end{bmatrix}  \, :\, \det{A}=1 \bigg\}.
  $$
Then one can see $T$ is indeed a maximal torus defined over $\Q$ that is neither $\Q$-split nor $\Q$-anisotropic.
Take $\displaystyle{
g_i=\begin{bmatrix}
   1 &    0  &   d_i    \\
         & 1 &  e_i  \\
         &       &  1\\
  \end{bmatrix} \in \SL_3(\R)}$ with $(d_i,e_i)$ diverging to infinity, then we have $(g_i)_*[\mu_{T}]=[\mu_{G}]$.\\
  
  \begin{proof}
We will check that $g_i$ diverges in $G(\R)/Z_G(S_0)(\R)$ for all $S_0$ nontrivial subtorus of $T$ defined over $\Q$.
Recall that the conjugacy class of a semisimple element in a semisimple linear algebraic group is closed and that $S_0(\Q)$ is dense in $S_0(\R)$. We have that $g_i$ diverges in $G(\R)/Z_G(S_0)(\R)$ iff there exists $s \in S_0(\Q)$, $g_i s g_i^{-1}$ diverges in $G(\R)$. 

Let us take 
$\displaystyle{
s= \begin{bmatrix}
  b &    c  &            \\
     pc    & b &    \\
         &       & a   
  \end{bmatrix} _{\neq id} \in T(\Q)}$. That means one of $a-1,c$ is nonzero, $(b^2-pc^2)a=1$ and $a,b,c\in\Q$. Note in addition that $\sqrt{p}\notin \Q$, we have $a\neq b\pm \sqrt{p} c$. 
  
Now $\displaystyle{g_i s g_i^{-1}=
\begin{bmatrix}
  b &    c  &     x_i       \\
     pc    & b &  y_i  \\
         &       & a   
  \end{bmatrix}}
$ with $\displaystyle{\begin{bmatrix}
  x_i         \\
   y_i   
  \end{bmatrix} = 
  \bigg(   \begin{bmatrix}
  a &    0           \\
   0 &   a    
  \end{bmatrix}
  -
  \begin{bmatrix}
  b &    c           \\
    pc    & b 
  \end{bmatrix}    \bigg) 
  \begin{bmatrix}
  d_i        \\
  e_i   
  \end{bmatrix}}
  $. 
  But $a\neq b\pm \sqrt{p} c$ implies that the matrix 
  $\displaystyle{
  \begin{bmatrix}
  a &    0           \\
   0 &   a    
  \end{bmatrix}-
  \begin{bmatrix}
     b   &   c           \\
    pc   &   b 
  \end{bmatrix}}
  $
  is invertible, so $(d_i,e_i)$ going to infinity is equivalent to $(x_i,y_i)$ going to infinity and hence $g_isg_i^{-1}$ diverges in this case.
\end{proof}

  \paragraph{Example 2}
Take $N=4$ and $(p,q)$ to be a pair of non-zero integers such that none of $(p,q,p/q)$ are squares of rational numbers.
Let $$
  T=\bigg\{ A=\begin{bmatrix}
  b  &    c  &       &     \\
  pc   & b  &    &  \\
     &       & d  & e\\
     &    &   qe  & d\\
  \end{bmatrix}\,:\, \det{A}=1
  \bigg\}.$$
  One can see this is also a maximal torus defined over $\Q$ that is neither $\Q$-split nor $\Q$-anisotropic.

Take $g_i= \begin{bmatrix}
  I_2  &    F_i    \\
  0  &  I_2    \\
  \end{bmatrix}$ where $F_i$ is a divergent sequence of $2$-by-$2$  matrices, then $(g_i)_*[\mu_{T}] \to [\mu_{G}]$.\\

\begin{proof}

It is clear here $l_0 =0$. If $S_0$ has nontrivial anisotropic part, according to the analysis in Section \ref{section6}, especially the construction of $K$(or $K'$), we must have a nontrivial extension of $\Q$ that simultanesously embeds into $\Q(\sqrt{p})$ and $\Q(\sqrt{q})$.
But one checks that $\Q(\sqrt{p}) \neq \Q(\sqrt{q})$ and they are Galois, hence  $S_0$ is $\Q$-split and the only candidate is $\{ A \in T \vert \,c=e=0 \}$.  Now it suffices to show that for $A= \diag(b,b,d,d)$ with $b^2d^2=1$ and $b \neq d$, we have $g_iAg_i^{-1}$ diverges.

Indeed $$\displaystyle{g_iAg_i^{-1}
= \begin{bmatrix}
  I_2  &    F_i    \\
  0  &  I_2
  \end{bmatrix}
 \begin{bmatrix}
  bI_2  &    0    \\
  0  &  dI_2
  \end{bmatrix}
 \begin{bmatrix}
  I_2  &    -F_i    \\
  0  &  I_2
  \end{bmatrix}
  =
  \begin{bmatrix}
  bI_2  &    (d-b)F_i     \\
  0  &  dI_2 
  \end{bmatrix}}
  $$
  diverges since $d-b\neq 0$.
\end{proof} 

\paragraph{Example 3} Let $N=4$ and 
   $$
  T=\bigg\{A= \begin{bmatrix}
  b  & c  &     &     \\
  2c & b  &      &  \\
     &    &  d   & e\\
     &    &  2e  & d\\
  \end{bmatrix}\, : \, \det{A}=1
  \bigg\}.$$ Then $T$ is a maximal $\Q$-torus of $\SL_4$ that is neither anisotropic over $\Q$ nor split over $\Q$. Let $S$ be the subgroup of $T$ with the additional requirement that $b=d$, $c=e$. Then it has a centralizer given by
  $$\displaystyle{ Z_G(S)=H=\bigg\{
  A= \begin{bmatrix}
  b  & c  &  f   &  h   \\
  2c & b  &  2h  &  f\\
  j  & l  &  d   & e\\
  2l & j  &  2e  & d\\
  \end{bmatrix}\, :\, \det{A}=1 \bigg\}
  }.$$
  Take $g_i= \begin{bmatrix}
  1  &    &  f_i  & h_i    \\
     & 1  &   2h_i & f_i  \\
     &    & 1  &  \\
     &    &    & 1 \\
  \end{bmatrix} \in Z_G(S)(\R)$ with $(f_i,h_i)$ diverges, then $(g_i)_*[\mu_{T}] \to [\mu_{H}]$.\\

  \begin{proof}
  Take arbitrary $A= \begin{bmatrix}
  b  & c  &     &     \\
  2c & b  &    &  \\
    &   &  d   & e\\
   &   &  2e  & d\\
  \end{bmatrix} \in T(\Q) \setminus S(\Q)$. So $b\neq d$ or $c \neq e$ and we need to show $g_iA g_i^{-1}$ diverges. 
  Indeed $\displaystyle{
  g_iA g_i^{-1}
  = 
  \begin{bmatrix}
  b  & c  &   x_{1i}  & x_{2i}    \\
  2c & b  &  x_{3i}  &  x_{4i}\\
    &   &  d   & e\\
   &   &  2e  & d\\
  \end{bmatrix}}
  $ with 
  $\displaystyle{
  \begin{bmatrix}
  x_{1,i}  & x_{2,i}    \\
    x_{3,i}  &  x_{4,i}\\
  \end{bmatrix}=
  \begin{bmatrix}
  d-b  & e-c    \\
   2(e-c)  &  d-b\\
  \end{bmatrix}
  \begin{bmatrix}
  f_i & h_i   \\
   2h_i  &  f_i\\
  \end{bmatrix}}
  $. 
  Note once $d-b$ or $e-c$ is nonzero, we would have 
  $\displaystyle{ \begin{bmatrix}
  d-b  & e-c    \\
   2(e-c)  &  d-b\\
  \end{bmatrix}}$ invertible as $\sqrt{2}$ is not a rational.
  Hence 
  $(x_{1,i}, x_{2,i},$ $x_{3,i},x_{4,i})$ diverges iff $(f_i,h_i)$ does. And we are done.
  \end{proof}

\section{Translates of reductive subgroups containing T}\label{section7}

In this section we prove Theorem \ref{theorem3} from the introduction. The proof will be the same as that in \cite{ShaZhe18} (see Section 10 therein) except that we work directly at the group level instead of Lie algebra. Recall $(S_0,S_1)$-clean says we are given a sequence $\{g_i\}$ in a reductive group $H_1(\R)$ containing a maximal torus $T$ and a subtorus $S_0$ of $S_1:=Z(H_1)$, the center of $H_1$, such that:
\begin{enumerate}
\item  $\{g_i\}$ is contained in $Z_G(S_0)(\R)$.
\item  $\{g_i\}$ diverges in $G(\R)/Z_G(S)(\R)$ for all $\Q$-subtorus $S$ of $S_1$ that strictly contains $S_0$.
\end{enumerate}

Let $T'$ be the $\R$-split part of $T$. Fix a maximal compact Lie subgroup $K$ of $G(\R)$ such that the corresponding Cartan involution preserves $T'$. We also fix a Weyl chamber $C$ in
$\Lie(T'(\R))$. 
For $i=0,1$, take a generic cocharacter $\alpha_i$ of $S_i$,  defined over $\R$. 
Thus they land in $T'$ and we assume they belong to $C$. 
Define $U_i$ to be the $\R$-subgroup of $G$ contracted by conjugate action of $\alpha_i$. 
So $U_0 \subset U_1$. 
Let $K_i = K \cap H_i(\R)$,  then we have the Iwasawa decomposition $H_i(\R)= K_i U_i(\R) T'(\R) $ for $i=0,1$.

Now according to this decomposition we may write $g_i=a_i u_i h_i$ with $a_i\in K_0$, $u_i \in U_0(\R)$, $h_i\in T'(\R)$ and so $(g_i)_*\mu_{H_1}=(a_iu_i)_*\mu_{H_1}$. 
Hence without loss of generality we may assume $g_i=u_i$, in other words $g_i$ is in $U_0(\R)$ (and in particular, is in $U_1(\R)$ ) in addition to being $(S_0,S_1)$-clean.

Let us also take another sequence $\{h_m\}$ that is $(S_1,T)$-clean and hence 
$\lim_m (h_m)_*$ $[\mu_{T}] = [\mu_{H_1}]$.

To show $(g_i)_*[\mu_{H_1}] \to [\mu_{H_{0}}]$, it suffices to show for $f_1,f_2\in C_c(G(\R)/\Gamma)$ and $(f_2, \mu_{H_{0}})\neq 0$,
\begin{equation*}
    \frac{(f_1, \mu_{H_{0}})}{(f_2, \mu_{H_{0}})}= 
    \lim_n \frac{(f_1, (g_n)_*\mu_{H_1})}{(f_2, (g_n)_*\mu_{H_1})} =
    \lim_n \lim_m \frac{(f_1, (g_nh_m)_*\mu_{T})}{(f_2, (g_nh_m)_*\mu_{T})}
\end{equation*}
if this fails, then we can find $\ep_2>0$ and $n_i\to \infty$ such that for all $n_i$,
\begin{equation*}
    \bigg\vert\frac{(f_1, \mu_{H_{0}})}{(f_2, \mu_{H_{0}})}- \lim_m \frac{(f_1, (g_{n_i}h_m)_*\mu_{T})}{(f_2, (g_{n_i}h_m)_*\mu_{T})}\bigg\vert \geq \ep_2 
\end{equation*}
Hence there exists a sequence $\{g_n\}$ in $U_1(\R)$ that is $(S_0,S_1)$-clean 
and $\{h_n\}$ that is $(S_1,T)$-clean
such that
$\lim_n(g_nh_n)_*[\mu_{T}] \neq [\mu_{H_{0}}]$. 
This is a contradiction against the following proposition.

\begin{prop}
 Assume we have a sequence $\{g_n\}$ in $U_1(\R)$ that is $(S_0,S_1)$-clean 
 and $\{h_n\}$ that is $(S_1,T)$-clean,
 then their product $\{g_nh_n\}$ is $(S_0,T)$-clean. As a result, 
 $\lim_n(g_nh_n)_*[\mu_T] = [\mu_{H_{0}}]$.
\end{prop}

\begin{proof}
Take $t\in T(\R)\setminus S_0(\R)$, we want to show $(g_nh_n)t(g_nh_n)^{-1}$ diverges.

Case I: $t\in S_1(\R)\setminus S_0(\R)$. \\
$(g_nh_n)t(g_nh_n)^{-1}=(g_n)t(g_n)^{-1}$ diverges by assumption on $g_n$.

Case II: $t\in T(\R) \setminus S_1(\R) $.\\
Write $t=t_0t_1$ with $t_1\in S_1(\R)$ and 
${t_0}_{(\neq id)} \in (H_1)_{ss}(\R) \cap T(\R)$, 
where $(H_1)_{ss}$ denotes the semisimple part of $H_1$. Then
\begin{equation*}
    \begin{aligned}
          (g_nh_n)t_0t_1(g_nh_n)^{-1} &= g_n h_n t_0 h_n^{-1} t_1  g_n^{-1} \\
         &= \big(g_n h_n t_0 h_n^{-1} t_1  g_n^{-1} (h_n t_0 h_n^{-1} t_1)^{-1}\big) (h_n t_0 h_n^{-1} t_1)
    \end{aligned}
\end{equation*}
The first parenthesis belongs to
$[U_1(\R), H_1(\R)]= U_1(\R)$
whereas the second is from $H_{1}(\R)$. Hence it diverges iff at least one of them diverges. But the latter one indeed diverges because of our assumption on $h_n$ and $t_0 \in T(\R) \setminus S_1(\R)$.
\end{proof}

\section{A counting problem}\label{section8}

In this section we apply the results obtained to a counting problem, see \cite{EskMozSha96}, \cite{Shah00}, \cite{OhSha14}, \cite{KelKon18}, \cite{ShaZhe18} for counting problems directly related to the one considered below.

Let $M/\Q$ be a number field and $\{\alpha_1,...,\alpha_{r_0}\}$ be its real embeddings and 
$\{\beta_1,...,\beta_{s_0}\}$ $\sqcup \{\overline{\beta_1},...,\overline{\beta_{s_0}}\} $ 
be conjugate pairs of complex embeddings. 
Let $m_0=r_0+2s_0$ be the extension degree of $M/\Q$.

Let $p(x)\in \OO_M[x]$ be fixed of degree $N$ with distinct roots in $\overline{\Q}$. Assume $p(x)=p_0(x)\cdot p_1(x)\cdot...\cdot p_{{a_0}}(x)$ in $M[x]$ where $p_0(x)$ splits in $M[x]$ of degree $l_0$ and for $i\geq 1$, $p_i(x)$ is irreducible in $M[x]$ of degree $l_i>1$. For a square matrix $A$, let $p_A(x):=\text{det}(xI-A)$ be its characteristic polynomial. 

Consider $X_p(\OO_M)=\{A\in M_{N}(\OO_M)\, \vert \,p_A(x)=p(x)  \}$. To state a meaningful counting problem, we use the geometric embedding

\begin{equation*}
\begin{aligned}
X_p'(\Z)=X_p(\OO_M)=& \{ (A_1,...,A_{r_0}, B_1,...,B_{{s_0}})\in M_N(\R)^{r_0} \times M_N(\C)^{s_0} \, \big\vert \\
& \,
\exists A\in X_p(\OO_M), A_i=\alpha_iA ,\, B_i= \beta_i A
\}
\end{aligned}
\end{equation*}

We also consider

\begin{equation*}
\begin{aligned}
X_p'(\R)=X_p(M\otimes_{\Q}\R)=
& \{ (A_1,...,A_{r_0}, B_1,...,B_{{s_0}})\in M_N(\R)^{r_0} \times M_N(\C)^{s_0} \,\big\vert\\
& p_{A_i}(x)=\alpha_i p(x) , \,p_{B_i}(x)= \beta_i p(x)
\}
\end{aligned}
\end{equation*}

For $R>0$, 
$B_R:=\{(A_1,...,B_{s_0})\in X'_p(\R) 
\,\big\vert \,\sum ||A_i||^2 + \sum ||B_i||^2 \leq R^2 \}$ where we let $||A||^2$ be $\sum|A_{i,j}|^2$.

\begin{thm}\label{thmcounting}
There is a constant $c_p>0$ such that 
 \begin{equation*}
    \lim_{R\to \infty} \frac{|X_p'(\Z) \cap B_R| }   {c_p R^{m_0 N(N-1)/2} (\log{R})^{l_0+a_0-1}} =1
 \end{equation*}
\end{thm}

Let $G=\SL_N$ and $G'=R_{M/\Q} G$. Then 
$G'(\R)= \SL_N(\R)^{r_0}\times \SL_N(\C)^{s_0}$
acts on $X_p'(\R)$ by conjugation. 
By linear algebra, one knows this action decomposes $X_p'(\R)$ into finitely many orbits, each of them defined over $\Q$.  Now we replace $X_p'(\R)$ by one of them, call it $Y_p'(\R)$, then $Y_p'(\R)\cong G'(\R)/T'(\R) $ for some maximal torus $T'$. 
By choosing the base point $x_0$ to be defined over $\Z$, we may assume $T'$ to be defined over $\Q$. By a theorem of \cite{BorHar62}, $X'_p(\Z)$ decomposes into finitely many(say, $c_1$) $G'(\Z)$-orbits. It remains to find the asymptotic of $|G'(\Z)\cdot x_0 \cap B_R|$ for each individual orbit.

To apply the equidistribution result as proved in previous sections, we need to further find the asymptotics of two quantities. The first one is the volume growth of $B_R$. For $T_0$ the stablizer of $x_0$ in $G'$, let $\mu_{G'/T_0}$ be the Haar measure on $G'(\R)/T_0(\R)$, then

\begin{prop}\label{VolumeAsymp}
   There is a constant $c_2>0$ such that 
   \begin{equation*}
       \lim_{R\to +\infty}  \frac{\mu_{G'/T_0}(B_R)  } {c_2 R^{m_0 N(N-1)/2}} =1
   \end{equation*}
\end{prop}

Indeed this is done in \cite{EskMozSha96}, Section 6 in the case $M=\Q$. The general case is just replacing $\SL_N(\R)$ by $\SL_N(\R)^{r_0}\times \SL_N(\C)^{s_0}$ and makes no difference. 

The second quantity is the volume of polytopes $\Omega^{'}_{g,\ep}$ averaging over $g\in B_R$. We will construct a subset $B'_{R,\ep}$ of $B_R$ such that for $g_R \in B'_{R,\ep}$ we have a good estimate of $\Vol(\Omega^{'}_{g_R,\ep})$. We need to show $B'_{R,\ep}$ is indeed big inside $B_R$. We first construct some $B_{R,\ep}$ which can be seen to be large easily and show $B'_{R,\ep}$ contains in $B_{R,\ep'}$ for some other $\ep'>0$ but decreses to $0$ as $\ep $ does so. And for those $g \in B_R \setminus B'_{R,\ep}$ we have a upper bound for $\Vol(\Omega^{'}_{g,\ep})$. 

In the case when the splitting field of $p(x)$ is totally real or when $M$ is a CM field, one can apply the computation in \cite{ShaZhe18} rather directly. Here, however, we need a few further computations(based on \cite{Shah00}) to make use of that. To keep the notation simple, we shall work with a single coordinate with the understanding that the full coordinate is just a tuple of everything that appears below. 
Let us emphasize that the analysis below is only necessary at those coordinates for real embeddings $\alpha$ of $M$ such that $\alpha(p(x))$ does not decompose completely over $\R$.

 \subsection{Conjugation over $\R$ and coordinates $(x_{i,j})$}

  Take $g_1 G'(\R)$ such that $x_1=g_1\cdot x_0$ takes the form $x_1=\begin{bmatrix}
    z_1 &      &      &      \\
         & z_2 &    &\\
         &       &  \ddots &  \\
         &      &      &  z_{b_0}  \\
  \end{bmatrix} $ with each $z_i$ being either a number( a real root of $\alpha p(x)$ ) or
  $z_i=\begin{psmallmatrix} u_i & -v_i\\ v_i & u_i\end{psmallmatrix}$(corresponding to a pair of complex roots $u_i \pm i v_i$ of $\alpha(p(x))$.). 
  
  Let us collect those
  
  \begin{equation*}
      \begin{aligned}
      \mathscr{I}'_1 :=& \{i\in \{1,...,N \} \,\vert\, (i,i)\text{-th entry is the (1,1)-entry of a 2-by-2 } z_{*}\}\\
      =& \{i\,,\, (x_1)_{i,i+1}\neq 0 \}
      \end{aligned}
  \end{equation*}
   Let $\mathscr{I}_1$ be the corresponding $i\in \{1,...,b_0\}$ such that $z_i$ is a $2$-by-$2$ matrix.
  
  The stablizer of $x_1$,
  $T_1:=g_1 T_0 g_1^{-1}$ consists of 
  $t_1=\text{diag}(d_1,d_2,...,d_{b_0})$ with each $d_i$ being 
  a number if $i\notin \mathscr{I}'_1$ 
  or a $2$-by-$2$ matrix of the form 
  $\begin{psmallmatrix} u & -v\\ v & u\end{psmallmatrix}$
  if $i\in \mathscr{I}_1$.

  Let $U_1$ be the subgroup of matrices with $d_i$ replaced by identity matrices of the same size and arbitrary entries above $\{d_i\}'$s. This is just those that are contracted by suitable elements from the $\R-$split part of $T_1$.
  
  For $w=(w_i)_{i\in \mathscr{I}_1}\in \R_{+}^{|\mathscr{I}_1|}$, let $h(w_i)=\begin{psmallmatrix} 1 & \sqrt{w_i}\\ 0 & 1\end{psmallmatrix}$ and 
  $h(w)$ be the matrix built on $t_1= \text{diag}(d_1,...,d_{b_0})$:
  \begin{enumerate}
  \item For $i\in \mathscr{I}_1$, replace $d_i$ by $h(w_i)$.
  \item For $i\notin \mathscr{I}_1$, replace $d_i$ by $1$.
  \end{enumerate}
   We have
  
  \begin{prop}[Proposition 3.2 from \cite{Shah00}]
  Under the bijective map
  $SO_n(\R)\times U_1(\R) \times  \R_{+}^{|\mathscr{I}_1|}\cong G'(\R)/T_{1}(\R)$ by 
  $(k,u,h(w)) \mapsto kuh(w)T_1$, 
  there is a constant $c_3>0$ such that,
  $\mu_{G'/T_1} = c_3 dk dw du$.
  \end{prop}
  
  Note compared to Proposition 3.2 from  \cite{Shah00}, the order of 
  $U_1$ and $\R_{+}^{|\mathscr{I}_1|}$ is reversed, which does not matter by computation in \cite{Shah00} above the Proposition 3.2, or one may also just check directly. From now on we identify $G'(\R)/T_1(\R)$ with $SO_n(\R)\times U_1(\R) \times  \R_{+}^{|\mathscr{I}_1|}$. 
  Hence $B_R = R_{g_1} \cdot 
  \{ (k,u,w) \,\vert\, ||kuh(w) \cdot x_1|| \leq R
  \}
  $
  
  Let us record here a computation. For each $i\in \mathscr{I}_1$,
  \begin{equation}\label{x-coord}
  \begin{aligned}
   h(w_i)z_ih(w_i)^{-1} =& \begin{psmallmatrix} 1 & \sqrt{w_i}\\ 0 & 1\end{psmallmatrix} 
    \begin{psmallmatrix} x_i & -y_i\\ y_i & x_i\end{psmallmatrix}
 \begin{psmallmatrix} 1 & -\sqrt{w_i}\\ 0 & 1\end{psmallmatrix}
\\
&=  \begin{psmallmatrix} x_i+y_i\sqrt{w_i} & -y_i(1+w_i)\\ y_i & x_i-y_i\sqrt{w_i}\end{psmallmatrix}
  \end{aligned}
  \end{equation}
  As $w_i$ grows, the upper right corner dominates the other entries.
 
 Now we turn to the definition of coordinates $x_{i,j}(i<j)$.
 
 Let $\pi$ be the projection of a matrix to its coordinates strictly above the diagonal. Look at the map $(u,w)\mapsto uh(w)\cdot x_1$ composed with $\pi$, i.e. $(u,w) \mapsto \pi(uh(w)x_1h(w)^{-1}u^{-1})$. By the above equation, one sees this is a bijection
 from $U_1(\R) \times \R_{+}^{|\mathscr{I}_1|} $ 
 onto 
 $\{(x_{i,j})_{i<j} \vert x_{i,j } \in \R ;\, x_{i,i+1}\in -v_i\R_{>1} (i\in \mathcal{I}'_1)\}$.
 Moreover, the measure $dwdu$ is pushed to $c_4 dx_{ij}(i<j)$ for some constant $c_4 >0$(see the computation in \cite{EskMozSha96} or \cite{Shah00}). 

\begin{lem}\label{first BRep}
 For any $\ep>0$,  we define 
 $$B_{R,\ep}:=R_{g_1} \cdot \{(k,u,w)\in R_{g_1}^{-1}B_R\, \vert\,
 |x_{i,i+1}|\geq \ep R(\forall i) \}.$$
  Then
 \begin{equation*}
 \lim_{\ep \to 0} 
 \liminf_{R>0}
 \frac{\mu_{G'/T_0}(B_{R,\ep})}{\mu_{G'/T_0}(B_R)} =1
 \end{equation*}
\end{lem}
 
 \begin{proof}
 This follows because Equation \ref{x-coord} shows the rest of the coordinates are dominated by those coming from $|x_{i,i+1}|$ . 
 \end{proof}
 
 \subsection{Diagonalization and coordinates $(y_{i,j})$}

  It is still not easy to work with elements from $B_{R,\ep}$ mainly because $x_1$ may not be diagonalized. We start by diagonalizing it. As before, we concentrate on some real coordinate $\alpha$ here.
  
  Consider $\delta = \diag(\delta_1, \delta_2, ...,\delta_{b_0})$.
  For $i\notin \mathscr{I}_1$, $\delta_i :=1$. 
  For $i \in \mathscr{I}_1$, $\delta_i := \frac{1}{\sqrt{-2i}}\begin{psmallmatrix} 1 & 1\\ i & -i\end{psmallmatrix}$. 
  This matrix has the property that 
  $\delta_i^{-1} \begin{psmallmatrix} u & -v\\ v & u\end{psmallmatrix} \delta_i = 
  \begin{psmallmatrix} -2(v+iu) & 0\\ 0 & 2(v-iu)\end{psmallmatrix} $ 
  and $\delta \in Z_{G'}(T_s)_{ss}(\C)$. 
  Hence,  $x_2 := \delta^{-1}\cdot x_1$ is a diagonal matrix who entries consisting of roots of $p(x)$ and $\delta^{-1} T_1 \delta = T_2$ is a subgroup of the diagonal matrices(not defined over $\R$).

 Consider the Iwasawa decomposition for $G'(\C)$ as a real Lie group only at real coordinates.
 $G'(\C)\cong (\SU(\R) N^{+}(\C) D(\C))^{r_0}$ 
 where $N^+$ consists of all unipotent upper triangular matrices,
 we write $uh(w)\delta = k_{u,w}n_{u,w} t_{u,w}$, then 
 \begin{equation*}
  \begin{aligned}
  B_R= R_{g_1}\cdot \{(k,u,w) \,\vert\, ||uh(w)\delta \cdot x_2|| \leq R\} = 
  R_{g_1}\cdot \{(k,u,w) \,\vert \,||n_{u,w} \cdot x_2|| \leq R\}
  \end{aligned}
 \end{equation*}
 
 Define a new coordinate $y_{i,j} =(n_{u,w} \cdot x_2)_{i,j} $, we let 
 \begin{equation*}
 B'_{R,\ep} := R_{g_1}\cdot \{ (k,w,u)\in R_{g_1}^{-1}B_R \, \vert\, |y_{i,i+1}| \geq \ep R \,(\forall i)  \} 
 \end{equation*}
 
 We wish to show that  
 \begin{prop}
  \begin{equation*}
 \lim_{\ep \to 0} \liminf_{R>0}
 \frac{\mu_{G'/T_0}(B'_{R,\ep})    } 
 { \mu_{G'/T_0}(B'_{R,\ep})  } 
 =1
 \end{equation*}
 \end{prop}
 
  \begin{proof}
  In virtue of Lemma \ref{first BRep}, it suffices to find $\ep'$ depending on $\ep$ but not depending on R such that 
   $B_{R,\ep'} \subset B'_{R,\ep}$.
 
      We need to relate $y$ to $x$. $(y_{i,j})= n_{u,w} \cdot x_2 = k^{-1}_{u,w} k_{u,w} n_{u,w} t_{u,w} \delta^{-1}\cdot x_1 = k^{-1}_{u,w} uh(w) \cdot x_1$. Recall $(x_{i,j})=uh(w) \cdot x_1$, so $y=k^{-1}_{u,w} x$ and $y_{i,i+1}= \sum_l (k^{-1}_{u,w})_{i,l}x_{l,i+1}$. 
      So we need to understand $k^{-1}_{u,w}$.
      
      First let $P_2:= N^{+} D$, then $k_{u,v} = u h(w) \delta$ mod $P_2(\C)$. Note $\delta$, $h(w)$ both normalize the group $U_1(\C)$. So $k_{u,w}$ takes the same form as $\delta$, i.e. $k_{u,w} = \text{diag}(k_1,...,k_{b_0})$.
      
      For $l \notin \mathscr{I}_1$, $k_l=1$ For $l\in \mathscr{I}_1$, let  
      $k^{-1}_l = \begin{psmallmatrix} 
      a_l & b_l\\ -\overline{b}_l & \overline{a_l}
      \end{psmallmatrix} $
      with $|a_l|^2 + |b_l|^2 =1$. 
      We know that $k^{-1}_l h(w_l) \delta_l $ 
      is an upper triangular matrix, in other words, 
      $$\frac{1}{\sqrt{-2i}}
      \begin{psmallmatrix} a_l & b_l\\ -\overline{b}_l & \overline{a_l}\end{psmallmatrix} 
      \begin{psmallmatrix} 1 &  \sqrt{w_l}\\ 0 & 1\end{psmallmatrix}
      \begin{psmallmatrix} 1 & 1\\ i & -i\end{psmallmatrix} = 
      \begin{psmallmatrix} * & *\\ 0& * \end{psmallmatrix}$$
      
      Solving these equations one get $\Re(b_l)-\sqrt{w_l}\Im(b_l)+\Im(a_l) =0$ and $\Im(b_l)+ \sqrt{w_l}\Re(b_l)+\Re(a_l) =0$,  implying $|b_l|=\frac{1}{1+w_l}|a_l|$. Note this implies $|a_l|\in (1/2,1)$. 
      
      Now we can compute $y_{i',i'+1}$,
      \[   
      y_{i',i'+1} = 
     \begin{cases}
       x_{i',i'+1} &\quad i'\notin \mathscr{I}'_1\cup \mathscr{I}'_1+1 \\
       a_{i}x_{i',i'+1} - \overline{b_i}x_{i'+1,i'+1}  &\quad i'\in \mathscr{I}'_1\\
        \overline{a_i}x_{i',i'+1} &\quad i'\in \mathscr{I}'_1 +1
     \end{cases}
     \]
     Note when $i' \in \mathscr{I}'_1$, $(k^{-1}_{u,w})_{i',i'}=a_i$ for some $i$ and when $i' \in \mathscr{I}'_1+1$, $(k^{-1}_{u,w})_{i',i'}=\overline{a_i}$ for some $i$.
     
     There is no issue with the first and the third case as $|a_i|>1/2$. We only need to show in the second case that there is $\ep'$,  $|x_{i',i'+1}|>\ep'R$ implies $|y_{i',i'+1}|>\ep R$ at least when $R$ is large enough.
     
     Recall the computation of $(h(w)x_1)_{i',i'+1}, (h(w)x_1)_{i',i'}$ made in Equation \ref{x-coord} and that 
     $x_{i',i'+1}=(uh(w)x_1)_{i',i'+1}=(h(w)x_1)_{i',i'+1}$, $x_{i'+1,i'+1}(uh(w)x_1)_{i'+1,i'+1}=(h(w)x_1)_{i'+1,i'+1}$ for $i'\in \mathscr{I}'_1$.
     If $(k,u,w)\in B'_{R,\ep'}$, then $|x_{i',i'+1}|= |y_i(1+w_i)|\geq \ep'R$. 
     
     \begin{equation*}
     \begin{aligned}
         |y_{i',i'+1}| &\geq |a_{i}x_{i',i'+1}| - |{b_i}x_{i'+1,i'+1}| \\
         &= 
         |a_i| |y_i(\frac{1}{w_i+1})| |w_i|- |b_i| |\frac{x_i}{\sqrt{w_i}} - y_i| \sqrt{w_i} \\
         &= w_i|a_i|\big( |y_i(1+\frac{1}{w_i})| - \frac{1}{\sqrt{w_i}(1+w_i)}
         |\frac{x_i}{\sqrt[]{w_i}} - y_i |
         \big)
     \end{aligned}
     \end{equation*}
     When $R$ is large enough(once $\ep'$ is fixed), we may assume
     $|1+w_i|\geq \max{\{\frac{|x_i+y_i|}{|y_i|},2 \}}$ so
     $\frac{1}{(1+w_i)}
         |\frac{x_i}{\sqrt[]{w_i}} - y_i | \leq |y_i|$ and 
         
         \begin{equation*}
     \begin{aligned}
         |y_{i',i'+1}| &\geq w_i |a_i| |y_i| (1-\frac{1}{\sqrt[]{2}})
         \geq R \ep' \frac{1}{2}|a_i|(1-\frac{1}{\sqrt[]{2}})
         \geq R \ep' \frac{1}{4}(1-\frac{1}{\sqrt[]{2}})
     \end{aligned}
     \end{equation*}
      Setting $\ep' := \frac{\ep}{\frac{1}{4}(1-\frac{1}{\sqrt[]{2}})}$  concludes the proof. The reader is reminded that $|y_i|$ are nonzero constants and $|a_i|\geq 1/2$.
   \end{proof}
   
   On the other hand, we need to compute the volume of the polytope.
   let $g=g_R = k_Ru_Rh(w_R)g_1 \in  B'_{R,\ep}$,
   
   \begin{equation*}
   \begin{aligned}
     &\text{Vol}(\Omega_{g,\ep'}) \\
     = & \Vol \{t\in \Lie (T_s(\R))\, \big\vert \, 
     ||g\exp{(t)}N_{M/\Q}(e_I) || \geq \ep' \,,\forall I \in \A_0     \}\\
        =& \text{Vol} \{t \in \text{Lie}( T_s(\R)) \, \big\vert\, 
        \exp{m_0\chi_I(t)}||k_{u,w}n_{u,w}t_{u,w}\delta g_1 N_{M/\Q}(e_I)|| 
        \geq \ep' \,,\forall I \in \A_0   \}   \\
        =& \text{Vol} \{t \in \text{Lie} (T_s(\R))  \, \big\vert \, 
        \exp{m_0\chi_I(t)} ||n_{u,w} N_{M/\Q}(e_I)|| 
        \geq \ep' \,,\forall I \in \A_0   \} 
   \end{aligned}
   \end{equation*}
    Then one can use Proposition 11.8 from \cite{ShaZhe18} to conclude:

\begin{prop}
There is a constant $c_4>0$, such that for $g_R \in B'_{R,\ep}$ for some fixed $\ep >0$, we have for any $\ep'>0$,
    \begin{equation*}
    \lim_{R\to \infty} \frac{\Vol(\Omega_{g_R,\ep'}) }{ c_4 (\log R)^{l_0+a_0-1} } =1
    \end{equation*}     
And for general $g_R \in B_R$, we have 
$\Vol(\Omega_{g_R,\ep}) = O ((\log R)^{l_0+a_0-1})$.
\end{prop}

\begin{proof}
Indeed, most computations in Section 11 of \cite{ShaZhe18} works through as long as  entries $\diag(\alpha_1,...,\alpha_n)$ are distinct and for $h \in N(\C)$ instead of just $N(\R)$ and with $N(\ep,R)$ there replaced by our $B'(R,\ep)$. Hence the conclusion of Proposition 11.8 in \cite{ShaZhe18} applies equally well here.
\end{proof}

\subsection{Conclude}
The rest of the proof is standard, (at least)  originating from \cite{DukRudSar93},\cite{EskMcM93}. 
   
\begin{proof}[Proof of Theorem \ref{thmcounting}]

By Proposition \ref{VolumeAsymp}, $\{B_R\}_R$ is a family of well-rounded sets. So to prove the counting asymptotics, it suffices to show there is a constant $c_p'>0$ such that if 
$$\psi_R(g\Gamma'):=\sum_{\Gamma'/\Gamma'\cap T_0(\R)} \chi_{B_R}(g\gamma\cdot x_0) ,$$
then 
$$ \frac{\psi_R}{c'_p R^{m_0N(N-1)/2} (\log{R})^{a_0+l_0-1}} \longrightarrow 1$$
weakly. 
So take $\phi \in C_c(G'(\R)/\Gamma')$ and $c_p'$ to be determined,

\begin{equation*}
    \begin{aligned}
     & (\psi_R,\phi) \\
     =& 
     \int_{G'(\R)/\Gamma'} \sum_{\Gamma'/\Gamma'\cap T_0(\R)} \chi_{B_R}(g\gamma\cdot x_0) \phi (g\Gamma') \hat{\mu_{T_0}(g\Gamma')}\\
    =&
    \int_{G'(\R)/\Gamma'\cap T_0(\R)} \chi_{B_R}(g\cdot x_0) \phi (g\Gamma'\cap T_0(\R)) \mu_{G'(\R)/\Gamma'\cap T_0(\R)}(g\Gamma'\cap T_0(\R)) \\
    = &
    \int_{B_R} \int_{T_0(\R)/T_0(\R)\cap \Gamma'}  \phi(g  tT_0(\R)\cap \Gamma') \mu_{T_0(\R)/T_0(\R)\cap \Gamma'}(tT_0(\R)\cap \Gamma') \mu_{G'/T_0}(gT_0(\R))\\
     = &
     \int_{B_R} \int_{G'(\R)/\Gamma'}  \phi(t \Gamma') g_*\mu_{T_0}(t\Gamma') \mu_{G'/T_0}(gT_0(\R))
    \end{aligned}
\end{equation*}
Now by propositions above,
$c_2R^{m_0N(N-1)/2}\sim \mu(B_R)$ 
and for any $\ep,\ep'>0$, $g\in B'(R,\ep)$, $c_4(\log{R})^{a_0+l_0-1} \sim \text{Vol}(\Omega'_{g_R,\ep'})$.
Also note that $g_R \in B'(R,\ep)$ satisfies the condition in Theorem \ref{theorem2}
for any choice of sequence $\{g_R\}$. If we define $c'_p=c_2c_4$ then for $\ep'$ small enough 
 (depending on the support of $\phi$), we have
\begin{equation*}
    \begin{aligned}
    &\frac{(\psi_R, \phi) }
    { c'_p R^{m_0N(N-1)/2} (\log{R})^{a_0+l_0-1}  }
     \\
     =& \frac{1}{\mu_{G'/T_0}(B'_{R,\ep})}
     \int_{B'_{R,\ep}} 
     \int_{\exp(\Omega_{g,\ep'})T'_a(\R)/\Gamma'} 
     \phi(t\Gamma') 
     \frac{g_*\mu_{T_0}(t)}{\Vol(\Omega_{g_R,\ep'})}
     \mu_{G'/T_0}(gT_0(\R))
     \\
     & +  \frac{1}{\mu_{G'/T_0}(B_R\setminus B'_{R,\ep})}
     \int_{B'_{R,\ep}} 
     \int_{\exp(\Omega_{g,\ep'})T'_a(\R)/\Gamma'} 
     \phi(t\Gamma') 
     \frac{g_*\mu_{T_0}(t)}{(\log{R})^{a_0+l_0-1}   }
     \mu_{G'/T_0}(gT_0(\R))
    \end{aligned}
\end{equation*}
Now the first term converges to $(\phi,1)$ but the $\lim_{\ep \to 0}\limsup_R$ of the absolute value of the second term converges to $0$. So we are done.

\end{proof}
  
  \subsection*{Acknowledgements}
  The author is grateful to Nimish Shah for suggesting the problem and for many useful suggestions.
We thank Cheng Zheng for a very careful explanation of their work and Uri Shapira for suggesting generalizing their work. We also thank Yongxiao Lin, Pengyu Yang, Lei Yang for helpful discussions.

\bibliographystyle{amsalpha}
\bibliography{ref}

\end{document}